\documentclass[10pt]{amsart}

\usepackage[left=3.4cm,right=3.4cm,top=2.6cm,bottom=2.6cm]{geometry}
\usepackage{graphicx,cite,mathrsfs,tikz,colortbl}
\usepackage{amssymb,amsthm}
\usepackage[colorlinks=true,urlcolor=blue,
citecolor=red,linkcolor=blue,linktocpage,pdfpagelabels,
bookmarksnumbered,bookmarksopen]{hyperref}
\usepackage[hyperpageref]{backref}

\makeatletter

\numberwithin{equation}{section}
\newtheorem{Prop}{Proposition}[section]
\newtheorem{Thm}[Prop]{Theorem}
\newtheorem{Lem}[Prop]{Lemma}
\newtheorem{Cor}[Prop]{Corollary}

\theoremstyle{definition}

\theoremstyle{remark}
\newtheorem{Rem}[Prop]{Remark}

\let\phi=\varphi

\newcommand{\abs}[1]{\mathopen| #1\mathclose|}
\newcommand{\bigabs}[1]{\bigl| #1\bigr|}
\newcommand{\norm}[1]{\mathopen\Vert#1\mathclose\Vert}
\newcommand{\intd}{{\,\mathrm{d}}}

\newcommand{\IR}{\mathbb{R}}
\newcommand{\IS}{{\mathbb S}}
\newcommand{\A}{{\mathscr A}}
\newcommand{\T}{{\mathcal T}}
\newcommand{\E}{\mathcal{E}}
\renewcommand{\L}{\mathcal{L}}
\newcommand{\limplies}{\Rightarrow}
\@ifundefined{leqslant}{}{\let\leq=\leqslant  \let\le=\leqslant}
\@ifundefined{geqslant}{}{\let\geq=\geqslant  \let\ge=\geqslant}
\@ifundefined{varnothing}{}{\let\emptyset=\varnothing}

\DeclareMathOperator{\Div}{div}
\DeclareMathOperator{\supp}{supp}
\DeclareMathOperator{\dist}{dist}
\DeclareMathOperator{\spanned}{span}
\DeclareMathOperator{\sign}{sign}

\definecolor{darkgreen}{rgb}{0, 0.7, 0}

\newenvironment{details}{%
  \setbox0=\vbox 
  \bgroup
  \color{gray}%
  \par
}{%
  \egroup
  \ignorespaces
}

\newcommand{\FIXME}[1]{%
  \textcolor{red}{ [\kern -0.38ex[#1]\kern -0.36ex]}\ \ignorespaces}

\makeatother

\title{Asymptotic symmetries for fractional  operators}

\author[C.~Grumiau, M.~Squassina and C.~Troestler]{}


\email{Christopher.Grumiau@umons.ac.be}
\email{marco.squassina@univr.it}
\email{Christophe.Troestler@umons.ac.be}

\thanks{The second author was supported by 2009 national MIUR project:
  ``Variational and Topological Methods in the Study of Nonlinear
  Phenomena''. The first and third author are partially supported by
  the program ``Qualitative study of solutions of variational elliptic
  partial differerential equations. Symmetries, bifurcations,
  singularities, multiplicity and numerics'' (2.4.550.10.F)
  and the project ``Existence and asymptotic behavior of solutions
  to systems of semilinear elliptic partial differential equations''
  (T.1110.14)
  of the \emph{Fonds de la Recherche Fondamentale Collective},
  Belgium.}

\begin{document}

\maketitle

\centerline{\scshape Christopher Grumiau}
\medskip
{\footnotesize
  \centerline{Universit\'e de Mons}
  \centerline{Institut Complexys}
  \centerline{D\'epartement de Math\'ematique}
  \centerline{Service d'Analyse Num\'erique}
  \centerline{Place du Parc, 20}
  \centerline{B-7000 Mons, Belgique}}

\medskip
\centerline{\scshape Marco Squassina}
\medskip
{\footnotesize
  \centerline{Dipartimento di Informatica}
  \centerline{Universit\`a degli Studi di Verona}
  \centerline{C\'a Vignal 2, Strada Le Grazie 15, I-37134 Verona, Italy}}

\medskip
\centerline{\scshape Christophe Troestler}
\medskip
{\footnotesize
  \centerline{Universit\'e de Mons}
  \centerline{Institut Complexys}
  \centerline{D\'epartement de Math\'ematique}
  \centerline{Service d'Analyse Num\'erique}
  \centerline{Place du Parc, 20}
  \centerline{B-7000 Mons, Belgique}}

\bigskip

\begin{abstract}
 In this paper, we study equations driven by a non-local
 integrodifferential operator $\L_K$ with homogeneous
 Dirichlet boundary conditions. More precisely, we study the
 problem
 \begin{equation*}
    \begin{cases}
      - \L_K u + V(x)u =  \abs{u}^{p-2}u, &\text{in }
      \Omega,\\
      u=0, &\text{in } \IR^N \setminus \Omega,
    \end{cases}
  \end{equation*}
  where $2 < p < 2^{*}_s = \frac{2N}{N-2s}$, $\Omega$ is an open
  bounded domain in $\IR^{N}$ for $N\ge 2$ and $V$ is a $L^\infty$
  potential such that $-\L_K + V$ is positive definite.  As a
  particular case, we study the problem
  \begin{equation*}
    \begin{cases}
      (- \Delta)^s u + V(x)u =  \abs{u}^{p-2}u, &\text{in }
      \Omega,\\
      u=0, &\text{in } \IR^N \setminus \Omega,
    \end{cases}
  \end{equation*}
  where $(-\Delta)^s$ denotes the fractional Laplacian (with $0<s<1$).
  We give assumptions on $V$, $\Omega$ and $K$ such that  ground state solutions
  (resp.\ least energy nodal solutions) respect
  the symmetries of some first (resp.\ second) eigenfunctions of
  $-\L_K + V$, at least for   $p$ close to $2$. We
  study the uniqueness, up to a
  multiplicative factor, of those types of solutions.
  The results extend those obtained for the
  local case.
\end{abstract}

\section{Introduction}

Non-local operators arise naturally in many different topics in
physics, engineering and even finance. For examples, they have applications
in  crystal dislocation, soft thin films, obstacle
problems~\cite{barlow,bass}, continuum mechanics~\cite{silling},
chaotic dynamics of classical conservative systems~\cite{stevens} and
graph theory~\cite{lovasz}. In this paper, we shall consider
the non-local counterpart of semi-linear elliptic
equations of the type
\begin{equation}
  \label{eq:pbmS}
  \begin{cases}
    -\Delta u + V(x)u= \abs{u}^{p-2}u, &\text{in } \Omega,\\
    u=0, &\text{in } \partial\Omega,
  \end{cases}
\end{equation}
where $\Omega$ is an open bounded domain with Lipschitz boundary,  $2
< p < 2^{*}$ is a
subcritical exponent (where $2^* := {2N}/({N-2})$ if $N \ge 3$, $2^* =
+\infty$ if $N=2$) and $V\in
L^{\infty}$ is such that $-\Delta + V$ is positive definite. Precisely, we
are predominantly interested in the
qualitative behaviour of solutions to
 \begin{equation}
   \label{eq:pbmF}
    \begin{cases}
      (- \Delta)^s u + V(x)u =  \abs{u}^{p-2}u, &\text{in }
      \Omega,\\
      u=0, &\text{in } \IR^N \setminus \Omega,
    \end{cases}
  \end{equation}
  where $(-\Delta)^s$ denotes the fractional Laplacian (with
$0<s<1$) and $2<p< 2^*_s := \frac{2N}{N-2s}$. Let us recall that,
up to a normalization factor,
$(-\Delta)^s$ may be defined~\cite{valdinoci3} as follows:
for $x\in \IR^N$,
\begin{equation*}
  (-\Delta)^su(x)
  := -c_{N,s} \lim_{\varepsilon \to 0} \int_{\complement B(x,\varepsilon)}
  \frac{u(y) - u(x)}{\abs{y-x}^{N+2s}} \intd y
  = - \tfrac{1}{2} c_{N,s} \int_{\IR^N}
  \frac{u(x+y) - 2 u(x) + u(x-y)}{\abs{y}^{N+2s}}\intd y
\end{equation*}
where $c_{N,s} := s 2^{2s} \, \Gamma(\frac{N+2s}{2}) / \bigl(\pi^{N/2} \,
\Gamma(1-s) \bigr)$ is a positive constant chosen \cite{Vazquez12}
to be coherent with the
Fourier definition of $(-\Delta)^s$.
This problem is variational and a ground state (resp.\ a least energy
nodal solution) can be defined from the associated Euler-Lagrange
functional --- see~\cite{valdinoci} (resp.\ Section~\ref{sec:setting})
for more details.
In this paper,  we would like to study the symmetries of those two types
of variational solutions. In fact, we consider a more general
setting: we are dealing with ground
state and least energy nodal solutions to the following equation:
\begin{equation}
  \label{eq:pbmP}
  \begin{cases}
    - \L_K u + V(x)u =  \abs{u}^{p-2}u, &\text{in }
    \Omega,\\
    u=0, &\text{in } \IR^N \setminus \Omega,
  \end{cases}
\end{equation}
where $\L_K$ is the non-local operator defined as follows
\begin{equation*}
  \L_Ku(x)
  := \int_{\IR^N} \bigl( u(x+y) - 2u(x) + u(x-y)\bigr)
  K(y)\intd y.
\end{equation*}
We shall assume that $K : \IR^N \setminus\{0\}\to
(0,+\infty)$ is a function such that $mK \in L^1(\IR^N)$ where $m(x) :=
\min \{\abs{x}^2,1\}$ and we require the existence of $\theta >0$ and
$s\in (0,1)$ such
that $K(x)\ge \theta \abs{x}^{-(N+2s)}$ for any $x\in
\IR^N\setminus\{0\}$. We also require that $K(x) = K(-x)$ for any
$x\in \IR^N\setminus\{0\}$.
In particular, we can consider
$K(x) = \tfrac{1}{2} c_{N,s}\abs{x}^{-(N+2s)}$ so that
$-\L_K$ is exactly the fractional
Laplacian operator $(-\Delta)^s$ as defined in~\eqref{eq:pbmS}
and \eqref{eq:pbmP} boils down to~\eqref{eq:pbmF}.

Let us point out that, in the current literature, there are several
notions of fractional Laplacian, all of which agree when the problems
are set on the whole $\IR^N$, but some of them disagree in a bounded
domain.  The values $(-\Delta)^s u(x)$ are, as we said, consistent
with the Fourier definition of $(-\Delta)^s$, namely ${\mathcal
  F}^{-1}(|\xi|^{2s}{\mathcal F} u)$ and also agree with the local
formulation due to Caffarelli-Silvestre~\cite{CS},
\begin{equation*}
  (-\Delta)^s u(x)
  = -C\lim_{t\to 0}\Bigl( t^{1-2s}\frac{\partial U}{\partial t}(x,t)\Bigr),
\end{equation*}
where $U:\IR^{N}\times (0,\infty)\to \IR$ is the solution to
$\Div(t^{1-2s}\nabla U)=0$ and $U(x,0)=u(x).$ The fractional laplacian
defined in this way is also called \emph{integral}.  In a bounded
domain $\Omega$, as in \cite{SV}, we choose to operate with it on
restrictions to $\Omega$ of functions defined on $\IR^N$ which are
equal to zero on $\complement\Omega$. A different operator
$(-\Delta)^s_{{\rm spec}}$ called \emph{regional, local} or
\emph{spectral} fractional Laplacian, largely utilized in literature,
can be defined as the power of the Laplace operator $-\Delta$ via the
spectral decomposition theorem.  Let $(\lambda_k)_{k \ge 1}$
and $(e_k)_{k \ge 1}$ be
the eigenvalues and eigenfunctions of $-\Delta$ in $\Omega$ with
Dirichlet boundary condition on $\partial\Omega$, normalized in such a
way that $\abs{e_k}_2=1$. Then, for every $s\in (0,1)$ and all $u\in
H^1_0(\Omega)$ with
\begin{equation*}
  u(x)=\sum_{k=1}^\infty \gamma_k e_k(x),\quad x\in\Omega,
\end{equation*}
one considers the operator
\begin{equation*}
  (-\Delta)^s_{\text{spec}} u(x)
  =\sum_{j=1}^\infty \gamma_j \lambda_j^s e_j(x),\quad x\in\Omega.
\end{equation*}
Of course, in this way, the eigenfunctions of
$(-\Delta)^s_{\text{spec}}$ agree with the eigenfunction $e_k$ of
$-\Delta$.  The operators $(-\Delta)^s$ and
$(-\Delta)^s_{\text{spec}}$ are different, in spite of the current
literature where they are sometimes erroneously interchanged.
In~\cite{CabreT}, the authors were able to recover also for the
spectral fractional Laplacian the aforementioned local realization
procedure.  We refer the interested reader to~\cite{SV2} for a careful
comparison of eigenvalues and eigenvectors of these two operators and
to~\cite{MontVerz} for further discussions about the correlations
among physically relevant nonlocal operators and the introduction of a
notion of fractional Laplacian for Neumann boundary conditions.

Under
the assumptions on $K$ stated above, the
Problem~\eqref{eq:pbmP} is variational (see
\cite[Section~2]{valdinoci}). The energy is defined on the space $H$
of Lebesgue measurable functions $g:
\IR^N\to \IR$ such that $g$ is zero almost everywhere outside
$\Omega$, its
restriction to $\Omega$ belongs to
$L^2(\Omega)$ and, furthermore,  the map $(x,y)\mapsto
\bigl(g(x)-g(y)\bigr) \sqrt{K(x-y)} \in
L^2\bigl(\IR^{2N}\setminus (\complement\Omega\times \complement\Omega)\bigr)$
(we write $\complement\Omega := \IR^N\setminus\Omega$).
The inner product of $H$ is defined as
\begin{equation}
  \label{scalarproduct}
  \langle u,v \rangle_H
  := \int_{Q} \bigl(u(x)-u(y)\bigr) \bigl(v(x)-v(y)\bigr)
  K(x-y)\intd x\intd y,
\end{equation}
where $Q:= \IR^{2N} \setminus (\complement\Omega \times \complement\Omega)$
(see e.g.\ \cite[Section~2]{valdinoci} for more details on $\langle \cdot,\cdot \rangle_H$).
The corresponding norm will be written~$\norm{\cdot}_H$.
The existence of ground state solutions has been proved
in~\cite{valdinoci} while the existence of least energy nodal
solutions is established in this paper
(see Section~\ref{sec:setting}).

We now state the main results.  For $k \ge 1$, we let
$\lambda_k$ (resp.\ $\phi_k$) be the $k^{\text{th}}$ eigenvalue
s counted without multiplicity (resp.\
eigenfunction) of the operator
$-\L_K + V$ with ``Dirichlet boundary conditions'' in $\Omega$
in the sense that $\phi_k =0 $ in
$\complement \Omega$.   We also consider $E_k$ the
eigenspace associated to $\lambda_k$.

\begin{Thm}\label{Thm-intro:simple}
  Assume that $-\L_K + V$ is positive definite.
  If $(u_{p})_{p>2}$ is a family of ground state (resp.\ least energy nodal)
  solutions to Problem~\eqref{eq:pbmP}, then
  \begin{equation*}
    \norm{u_p}_{H} + \abs{u_p}_{2} \le C \lambda_i^\frac{1}{p-2},
    \qquad \text{$\lambda_i = \lambda_1$ (resp.\ $\lambda_2$).}
  \end{equation*}
  If $p_{n}\to 2$ and
  $\lambda_{i}^{\frac{1}{2-p_{n}}}u_{p_{n}}\rightharpoonup u_{*}$ in
  $H$ (the weak convergence necessarily holds, up to a subsequence),
  then $\lambda_{i}^{\frac{1}{2-p_{n}}}u_{p_{n}}\to u_{*}$ in $H$
  and $u_{*}\ne 0$ satisfies
  \begin{equation*}
    \left\{
      \begin{aligned}
        -\L_K u_{*} + V u_* &=\lambda_{i}u_{*},&& \text{ in } \Omega ,\\
        u_{*}&=0,&&\text{ in }\IR^N \setminus \Omega.
      \end{aligned}
    \right.
  \end{equation*}
  Assume that $\lambda_{1}$ (resp.\ $\lambda_2$) is simple. Then, for
  $p$ close to $2$ and any reflection $R$ such that
  $R(\Omega)=\Omega$, ground state solutions (resp.\ least energy
  nodal solutions) to Problem~\eqref{eq:pbmP} possess the same symmetry or
  antisymmetry as $\phi_{1}$ (resp.\ $\phi_2$) with respect to
  $R$. Moreover, this type of solution is unique
  up to its sign.
\end{Thm}

The proof of the previous Theorem makes use of the
implicit function theorem (see Sections~\ref{sec:a priori}
and~\ref{sec:IFT}).
In particular, since it is known that $\phi_1$ is a positive
eigenfunction when $V\equiv 0$ (see~\cite[Proposition 9, assertion
c)]{valdinoci}) and thus $\lambda_1$ is simple.
As we show in Lemma~\ref{dimE1=1}, $\lambda_1$ is also simple
when $V \in L^\infty$ and $-\L_K + V$ is positive definite.
In these cases, we have the
following

\begin{Cor}\label{Thm-intro:cor}
  Assume that $-\L_K + V$ is positive definite.
  If $(u_{p})_{p>2}$ are ground state
  solutions to
  \begin{equation}
    \label{Vnullo}
    \begin{cases}
      - \L_K u +Vu=  \abs{u}^{p-2}u, &\text{in }
      \Omega,\\
      u=0, &\text{in } \IR^N \setminus \Omega,
    \end{cases}
  \end{equation}
  then, for $p$ close to $2$ and any reflection $R$ such that
  $R(\Omega)=\Omega$, ground state solutions of \eqref{Vnullo} possess
  the same symmetry or antisymmetry as  $\phi_{1}$ with respect to
  $R$. Moreover, this type of solution is unique up to its sign.
\end{Cor}

When $\lambda_2$ is not simple, then
we cannot use the implicit
function theorem. In this case, we just are able to conclude the
following result which no longer asserts uniqueness
(see Section~\ref{sec:laypunov}).

\begin{Thm}\label{Thm-intro:simple3}
  Assume that $-\L_K + V$ is positive definite
  and that the zero set of any function in $E_2\setminus\{0\}$ has zero
  Lebesgue measure.
  For $p$ close to $2$, least energy
  nodal solutions $u_p$ to Problem \eqref{eq:pbmP} possess the same
  symmetries and antisymmetries of their orthogonal projection
  on~$E_2$.
\end{Thm}

Note that the assumption about the zero sets are known to hold
when $-\L_K=(-\Delta)^s$ (see \cite{felli}).
We are also able to localize least energy nodal solutions when $p
\approx 2$, see Theorem~\ref{thm:localization}.

These theorems are a generalization of the corresponding results
for the semi-linear case~\eqref{eq:pbmS}.
In~\cite{bbg,bbgv,g,gt}, this equation has been
extensively studied for several choices of potentials~$V$, boundary
conditions and non-linearities. We also refer the reader to the
references therein.

In this paper, $\abs{\cdot}_p$  will denote the
traditional norm in $L^p$. The notation $f'(u)[v]$ stands for  the
Fr\'echet derivative of
the function $f$ at $u$ in the direction $v$.

\section{Formulation, ground state and
  least energy nodal solutions}
\label{sec:setting}

\noindent
In~\cite{valdinoci}, it is proved that $H$ is a Hilbert space when
endowed with the norm
\begin{equation*}
  \norm{u}^2_H = \int_{Q} \abs{u(x) - u(y)}^2 K(x-y)\intd x
  \intd y.
\end{equation*}
Moreover, the embedding $H \hookrightarrow L^q(\IR^N)$ is continuous
for $q\in [1, 2^*_s]$ and it
is compact when  $q < 2^*_s$. In particular, the Sobolev's
inequality holds: there exists $C>0$ such that, for any
$u\in H$, $\abs{u}_{2^*_s}\le C \norm{u}_H$.

At this point, we may define the functional
\begin{equation}
  \label{energy}
  \E_p : H\to \IR : u \mapsto \frac{1}{2} \norm{u}^2_{H} +
  \frac{1}{2}\int_\Omega V(x)u^2 \intd x
  - \frac{1}{p} \abs{u}_{p}^p.
\end{equation}
whose corresponding Euler-Lagrange equation is the weak formulation
of Problem~\eqref{eq:pbmP} (see~\cite{valdinoci2}):
for $u\in H$,
\begin{equation*}
  \forall \phi\in H,\quad
  \int_{\IR^{2N}} \bigl(u(y)-u(x)\bigr) \bigl(\phi(y)-\phi(x)\bigr)
  K(y-x) \intd y\intd x
  + \int_{\Omega} V(x) u\phi \intd x
  = \int_{\Omega} \abs{u}^{p-2}u\phi \intd x.
\end{equation*}
To establish the existence of ground state solutions, R.~Servadei and
E.~Valdinoci~\cite{valdinoci} assume that $-\L_K + V$ is positive
definite and make use of the traditional Mountain Pass Theorem.  They
minimize~$\E_p$ on the following Nehari manifold
\begin{equation}
  \label{nehari}
  \mathcal{N}_p := \Bigl\{ u\in H\setminus\{0\} :
  \int_{\IR^{2N}} \bigl(u(x)-u(y)\bigr)^2  K(x-y) \intd x\intd y
  + \int_{\Omega} V(x) u^2 \intd x
  = \int_{\Omega} \abs{u}^{p} \intd x \Bigr\}.
\end{equation}
Up to our knowledge, there is no characterization of  sign-changing solutions
mentioned in the literature for the non-local case. 
We prove hereafter that a least
energy nodal solution exists and may be
characterized as a minimum of $\E_p$ on the traditional nodal Nehari set
\begin{equation}
  \label{nodalnehari}
  \mathcal{M}_p := \bigl\{ u\in H\setminus\{0\} : u^\pm\ne 0 \text{ and
  }\E_p(u) = \max_{t^+,\ t^- >0} \E_p (t_p^+
  u^+ + t_p^- u^-) \bigr\},
\end{equation}
where $u^+ := \max\{u,0\}$ and $u^- := \min\{u,0\}$.
Let us remark that an important technical difference with the
classical local semilinear occurs: $\langle u^+,u^-\rangle_H = -2
\int_\Omega\int_\Omega u^+(x) u^-(y) K(x-y)\intd x\intd y > 0$ for any
sign-changing solutions $u\in H$.

Let us mention that $u \in \mathcal{M}_p$ is equivalent to
$u^\pm\ne 0$ and
$\E_p'(u)[u^{\pm}]=0$.  Indeed,  it is clear that if $u\in
\mathcal{M}_p$ then
$\E_p'(u)[u^{\pm}]=0$. For the other direction, let us first remark
that, for any
sign-changing function $u\in
H$, there exists $t^+ >0$ and $t^->0$ such that $t^+
u^+ + t^- u^-\in \mathcal{M}_p$.\label{proj-nodal-Nehari}
This comes easily from the fact that $\E(u) \to -\infty$ when
$\norm{u}_H \to +\infty$ while being constrained to any finite
dimensional subspace, and the fact that the maximum cannot be on the
boundary of the cone because $\E'_p(u^+)[u^-] = \langle u^+, u^-
\rangle_H > 0$ and $\E'_p(u^-)[u^+] > 0$.
It is thus sufficient to show that the system
$\E_p'(t^+u^+ + t^- u^-)[u^{\pm}]=0$ has at most one non-trivial solution
$(t^+,t^-)$ with $t^\pm >0$.
Developing the equations
$\E_p'(t^+u^+ + t^- u^-)[u^{\pm}]=0$ leads to a system of the type
\begin{equation*}
  \begin{cases}
    t^- = A (t^+)^{p-1} - B t^+ ,\\
    t^+ = C (t^-)^{p-1} - D t^-,
  \end{cases}
\end{equation*}
for some $A,B,C,D >0$.  As a quick drawing will convince you, one can
see the solutions of this system as the intersection of two increasing
functions of~$t^+$, one that it super-quadratic and one that is
sub-quadratic.  Hence a single intersection exists.

\begin{details}
  Let $\phi(t^+) := A (t^+)^{p-1} - B t^+$ (super-quadratic) and
  $\psi^{-1}(t^-) := C (t^-)^{p-1} - D t^-$ (on the set where $C
  (t^-)^{p-1} - D t^- > 0$, where the expression is increasing w.r.t.\
  $t^-$, hence invertible and $\psi$ is sub-quadratic).  At a $t^+ >
  0$ s.t.\ $\phi(t^+) = \psi(t^+)$, one wants to prove that
  $\partial\phi(t^+) > \partial\psi(t^+)$ i.e.,
  \begin{equation*}
    (p-1) A (t^+)^{p-2} - B
    = \partial\phi(t^+)
    > \frac{1}{\partial\psi^{-1} \circ \psi(t^+)}
    = \frac{1}{(p-1)C (t^-)^{p-2} - D} \Bigr|_{t^- = \psi(t^+)}
  \end{equation*}
  i.e.,
  \begin{equation*}
    \bigl( (p-1) A (t^+)^{p-2} - B \bigr)
    \bigl( (p-1)C (t^-)^{p-2} - D \bigr) > 1.
  \end{equation*}
  Multiplying by $t^+ t^-$ and using the fact that $(t^+, t^-)$ is a
  solution of the system, one can rewrite it as
  \begin{equation*}
    \bigl( (p-1) t^- + (p-2)B t^+ \bigr)
    \bigl( (p-1) t^+ + (p-2)D t^- \bigr) > t^+ t^-.
  \end{equation*}
  This holds because all terms are positive and one of these terms is
  $(p-1) t^- (p-1) t^+$ which is $> t^+ t^-$ because $p-1 > 1$.
\end{details}

Note that, contrarily to the local case, it is not
true that  $u^\pm \in
\mathcal{N}_p$ if and only if $u^+ + u^-\in
\mathcal{M}_p$.
This is again a consequence of the fact that
$\langle u^+,u^-\rangle_H \ne 0$ for sign-changing functions.

Now let us show that \eqref{eq:pbmP} possesses at least one least
energy nodal solution.  In doing so, we will prove again, as a
byproduct, that non-negative and non-positive solutions also exist.
We take our inspiration from~\cite{Bartsch-Liu-Weth04}.

Let $\norm{\cdot}$ be the norm on $H$, equivalent to $\norm{\cdot}_H$
(see Section~\ref{sec:a priori}),
induced by the inner product
\begin{equation}
  \label{eq:equiv-inner-prod}
  \langle u, v\rangle
  = \langle u, v\rangle_H + \int_\Omega V(x)uv \intd x.
\end{equation}
We know that $\nabla{\mathcal E}_p(u)=u-A(u)$, namely
\begin{equation*}
  \E_p'(u)[\varphi]
  = \langle u-A(u) , \varphi\rangle,
  \qquad \langle A(u), \varphi \rangle :=\int_\Omega |u|^{p-2}u\varphi,
  \qquad u,\varphi\in H.
\end{equation*}
If $H^\pm$ denote the positive and negative cones of $H$, we set
$$
H^\pm_\varepsilon:=\big\{u\in H: {\rm dist}(u,H^\pm)<\varepsilon\big\}.
$$
Then, we first state the following

\begin{Lem}[Order perserving property]
  \label{max-princ}
  Assume that $-\L_k + V$ is positive definite
  and let $u\in H$ be such that
  \begin{equation}
    \label{disV}
    \langle u, \phi \rangle \ge 0,
    \quad\text{for every $\varphi\in H$ with $\varphi\geq 0$}.
  \end{equation}
  Then $u\geq 0$.
\end{Lem}
\begin{proof}
  Testing inequality~\eqref{disV} with $\phi = - u^-\in H^+$ yields
  \begin{equation*}
    \langle u, -u^- \rangle
    = - \langle u^+, u^- \rangle
    - \norm{u^-}^2
    \ge 0.
  \end{equation*}
  As $\langle u^+, u^- \rangle = \langle u^+, u^- \rangle_H \ge 0$,
  we get $\norm{u^-}^2 = 0$ and thus $u^- = 0$ since
  $-\L_k + V$ is positive definite.
\end{proof}

\begin{Lem}
\label{charact}
For every $\varepsilon>0$ sufficiently small, 
  $A(\partial H^\pm_\varepsilon)\subseteq H^\pm_\varepsilon.$
  In particular, if $u\in H^\pm_\varepsilon$ is a critical point
  of $\E_p$,
namely $A(u)=u$, then $u\in H^\pm.$
\end{Lem}
\begin{proof}
We have, for $2\leq p\leq 2^*_s$,
\begin{equation}
\label{prima}
  \forall u \in H,\qquad
  \abs{u^+}_p = \min_{w\in H^-}\abs{u-w}_p
  \le C \min_{w\in H^-}\|u-w\| = C \dist(u,H^-).
\end{equation}
  Let $P_+$ denotes the metric projector on the positive cone $H^+$
  for the norm~$\norm{\cdot}$.
  The metric projection on the convex set $H^+$ is
  characterized by
  \begin{equation}
    \label{eq:proj-convex}
    \forall \phi \in H^+,\qquad
    \langle u - P_+u,\, \phi - P_+ u \rangle \le 0.
  \end{equation}
  Because $H^+$ is a cone pointed at~$0$, this is equivalent to
  \begin{equation}
    \label{eq:proj-cone}
    \langle u - P_+u,\, P_+ u \rangle = 0
    \qquad\text{and}\qquad
    \forall \phi \in H^+,\quad
    \langle u - P_+u,\, \phi \rangle \le 0.
  \end{equation}
  The implication $\eqref{eq:proj-cone} \Rightarrow
  \eqref{eq:proj-convex}$ is obvious.  For $\eqref{eq:proj-convex}
  \Rightarrow \eqref{eq:proj-cone}$, taking $\phi = t P_+u$ with $t >
  0$ in~\eqref{eq:proj-convex} yields $(t-1)\langle u - P_+u,\, P_+ u
  \rangle \le 0$, whence $\langle u - P_+u,\, P_+ u \rangle = 0$.

  Consequently, if we set
  $P_-u := u-P_+u$, then $P_- u$ is orthogonal to $P_+u$. 
  Moreover, since $\langle P_-u, \varphi\rangle\leq 0$
  for every $\varphi\in H^+$, it follows that
  $P_- u\leq 0$ by virtue of Lemma~\ref{max-princ}. If $v:=A(u),$ then
  taking inequality \eqref{prima} into account,
  \begin{align*}
    \dist\bigl(A(u),H^-\bigr) \norm{P_+v}
    &\le \norm{v-P_- v} \norm{P_+v}
    = \norm{P_+v}^2 \\
    & = \langle v, P_+v\rangle
    =\int_{\Omega} \abs{u}^{p-2}u \, P_+v
    \le \int_{\Omega} \abs{u^+}^{p-2}u^+ \, P_+v \\
    &\le \abs{u^+}_p^{p-1} \, \abs{P_+ v}_p
    \le C \dist(u,H^-)^{p-1} \norm{P_+v}.
  \end{align*}
  If $\dist(u,H^-)$ is small enough, we have 
  $\dist(A(u),H^-) \le \frac{1}{2} \dist(u,H^-)$,
concluding the proof.
\end{proof}

According to~\cite[Lemma 3.2]{LiuSun},
Lemma \ref{charact} implies the existence of a pseudo-gradient vector 
field ${\mathcal G}$ such that $H^\pm_\varepsilon$ are (forward) invariant for the
descending flow. Let us denote the flow by $\eta$ i.e., $\eta(\cdot,
u)$ is
the maximal solution to
\begin{equation*}
\begin{cases}
    \partial_t \eta(t,u) = - {\mathcal G}\bigl(\eta(t,u)\bigr), \\
\eta(0,u)=0,
\end{cases}  
\end{equation*}
defined on the interval $\bigl[0,T(u)\bigr)$.
It follows that $\partial H^\pm_\varepsilon\subseteq
\A(H^\pm_\varepsilon)$, where $\A(H^\pm_\varepsilon)$
stands for the basin of attraction of $H^\pm_\varepsilon$
for the flow~$\eta$.

First, we need the following
\begin{Lem}
  \label{basin-of-0}
For $\varepsilon>0$ sufficiently small, $\overline{H^+_\varepsilon}\cap
  \overline{H^-_\varepsilon} \subseteq {\mathscr A}_0,$
  where ${\mathscr A}_0$ is
  the basin of attraction of $0$.  In particular,
  $\E_p(u) > 0$ for every
$u\in \overline{H^+_\varepsilon}\cap \overline{H^-_\varepsilon}\setminus\{0\}.$
\end{Lem}
\begin{proof}
  We know that, for $\epsilon > 0$ small enough,
  $H^+_\varepsilon \cap H^-_\varepsilon$ is (forward) invariant for $\eta$
  and the sole critical point it
  contains is $0$.  The map $\eta$ being a pseudo-gradient flow
  and the Palais-Smale condition is satisfied, either
  $\E_p\bigl(\eta(t,u)\bigr) \to -\infty$ as $t\to T(u)$ or
  $\eta(t,u)$ possesses a limit point $u^*$ as $t \to T(u)$
  which is a critical point of
  $\E_p$, in which case $T(u) = +\infty$.  Also, if
$u\in H^+_\varepsilon \cap H^-_\varepsilon$, then the second case rephrases as
$\eta(t,u)\to 0$ as $t\to T(u)$. To conclude, we need to rule out the first case. 
Taking into account inequalitites \eqref{prima}, it follows that
  \begin{equation*}
    \E_p(u) \ge -\frac{1}{p}\abs{u}_p^p
    \ge -\frac{1}{p}\bigl( \abs{u^+}_p + \abs{u^-}_p \bigr)^p
    \ge -\frac{C}{p} \bigl(\dist(u,H^-) + \dist(u,H^-)\bigr)^p
    \ge -\frac{C \, (2\varepsilon)^p}{p},
  \end{equation*}
whenever $u\in H^+_\varepsilon\cap H^-_\varepsilon$.
  Finally, the flow decreases the functional $\E_p$, so
  $\E_p(u) > \E_p\bigl(\eta(t,u)\bigr) > \E_p(0) = 0$
  for all $t > 0$.
\end{proof}

We now state the following
\begin{Thm}
  \label{exists-sign-changing}
There exist a solution to non-negative, a non-positive and a sign-changing solution to Problem \eqref{eq:pbmP}.
\end{Thm}
\begin{proof}
  There exists a solution to Problem~\eqref{eq:pbmP}
  in $H^+_\varepsilon\setminus \overline{H^-_\varepsilon}$, a solution 
in $H^-_\varepsilon\setminus \overline{H^+_\varepsilon}$ and one solution in $H\setminus (\overline{H^+_\varepsilon}\cup
\overline{H^-_\varepsilon})$. This follows directly from \cite[Theorem 3.2]{LiuSun}, provided that
  one shows the existence of a path $h:[0,1]\to H$ such that
$h(0)\in H^+_\varepsilon \setminus  H^-_\varepsilon$, 
$h(1)\in H^-_\varepsilon \setminus H^+_\varepsilon$ and 
  \begin{equation*}
    0
    = \inf_{\overline{H^+_\varepsilon}\cap \overline{H^-_\varepsilon}} \E_p
    > \sup_{t\in [0,1]} \E_p\bigl(h(t)\bigr).  
  \end{equation*}
It is readily seen that for any finite dimensional subspace $E$ of $H$, 
there exists $R>0$ such that $u\in E$ and $\|u\|\geq R$ imply 
  that ${\mathcal E}_p(u)<0$.
  Pick $E:=\spanned\{u_0,u_1\}\subseteq H$,
  where $u_0\in H^+\setminus\{0\}$ and 
  $u_1 \in H^-\setminus\{0\}$ are non-colinear given elements.
  If $R > 0$ is the corresponding radius, set
  \begin{equation*}
    h(t) := R^* \bigl((1-t) u_0+t u_1\bigr)\in H\setminus\{0\},
    \quad t\in [0,1],
  \end{equation*}
  where $R^*$ is large enough
  so that $\min\bigl\{\norm{h(t)} :t\in [0,1] \bigr\}\ge R$.
  This ends the proof, thanks to Lemma~\ref{charact}.
\end{proof}

\noindent
Finally, we state the following
\begin{Thm}
  There exists a sign-changing solution $u^*$ to~\eqref{eq:pbmP}
  with minimal energy among all sign-changing solutions.
  In addition, $u^* \in \mathcal{M}_p$ achieves the minimum of
  $\E_p$ on~$\mathcal{M}_p$.
\end{Thm}
\begin{proof}
  Lemma~\ref{charact} says that the only critical points of $\E_p$
inside $H^+_\varepsilon\cup H^-_\varepsilon$ are those belonging to $H^+\cup H^-$.
Thus, all the sign-changing critical points of ${\mathcal E}_p$ belong
to the closed set $H\setminus(H^+_\varepsilon\cup H^-_\varepsilon).$ Let us set
  \begin{equation*}
    c := \inf\{\E_p(u) : u \text{ is a sign-changing critical point of }
    \E_p \}.
  \end{equation*}
  Since Theorem~\ref{exists-sign-changing} says that the set is non-empty,
  there exists a sequence $(u_n)\subseteq H\setminus
  (H^+_\varepsilon\cup H^-_\varepsilon)$ such that
  $\E_p(u_n)\to c$, as $n\to\infty$.  Since $\E_p$ satisfies
the Palais-Smale condition, it follows that $(u_n)$ admits a subsequence
which converges strongly in $H$ to a limit point $u^*\in H\setminus(H^+_\varepsilon\cup H^-_\varepsilon),$
  such that $\E_p'(u^*)=0$ and $\E_p(u^*)=c$.

  As $u^*$ changes sign and $\E_p'(u^*)[u^\pm] = 0$, $u \in
  \mathcal{M}_p$ (see the discussion following~\eqref{nodalnehari}).
  To show that $u^*$ has minimal energy on $\mathcal{M}_p$,
  pick any $v \in \mathcal{M}_p$ and consider the rectangle
  $C := \{ t^+ v^+ + t^- v^- : t^+, t^- \in [0,R]\}$
  where $R$ is large enough so that $\E_p(t^+ v^+ + t^- v^-) < 0$
  whenever $t^+ = R$ or $t^- = R$.
  We will show that there exists $w \in C \setminus
  (H^+_\varepsilon\cup H^-_\varepsilon)$ such that $\bigl(
  \E_p(\eta(t, w)) \bigr)_{t \ge 0}$ is bounded from below.
  This will conclude the proof because, $\eta$ being a gradient flow,
  $\eta(t, w)$ must then posses a limit point $w^* \in H \setminus
  (H^+_\varepsilon\cup H^-_\varepsilon)$ which is a sign-changing
  critical point of~$\E_p$.  Thus $\E_p(v) \ge \E_p(w) \ge
  \E_p(\eta(t,w)) \ge \E_p(w^*) \ge c$ for all $t \in [0,+\infty)$.

  To prove the existence of~$w$, we follow an argument similar to
  the last one of the proof of Theorem~3.1 in~\cite{LiuSun}
  which we shall briefly explain.
  Let us start by considering $\A_0$, the basin of attraction of~$0$.
  The set $O := C \cap \A_0$ is a non-empty open subset in~$C$
  on which $\E_p \ge 0$.  By the choise of~$R$,
  $\bigl\{ t^+ v^+ + t^- v^- : (t^+, t^-) \in (\{R\} \times [0,R])
  \cup ([0,R] \times \{R\}) \bigr\} \cap O = \emptyset$.
  Consequently, there exists a connected component $\Gamma$ of
  $\partial O$ intersecting both $[0,R] v^+$ and $[0,R] v^-$.
  Thus $\Gamma \cap H^+ \ne \emptyset$ and $\Gamma \cap H^- \ne \emptyset$.
  Let $\A(H^+) = \A(H^+_\epsilon)$ (resp.\ $\A(H^-) =
  \A(H^-_\epsilon)$) be the basin of attraction of~$H^+$
  (resp.\ $H^-$), where $\epsilon$ is small enough.
  The sets $\Gamma \cap \A(H^+)$ and $\Gamma \cap \A(H^-)$ are
  non-empty open subsets of~$\Gamma$.
  Moreover they are disjoint because, if they weren't,
  Lemma~\ref{basin-of-0} would imply that $\Gamma \cap \A_0 \ne 0$
  but, on the other hand,
  $\Gamma \subseteq \partial O \subseteq \partial \A_0$ implies
  $\Gamma \cap \A_0 = \emptyset$.
  In conclusion there exists $w \in \Gamma \setminus \bigl(\A(H^+)
  \cap \A(H^-) \bigr)$.  It remains to show that
  $\bigl( \E_p(\eta(t, w)) \bigr)_{t \ge 0}$ is bounded from below.
  But this is clear because $\partial \A_0$ is forward invariant
  and, thanks again to Lemma~\ref{basin-of-0},
  $\E_p(u) > 0$ for all $u \in \partial \A_0$.
\end{proof}

\begin{Rem}
  Following the ideas
  of~\cite[proposition~3.1]{Bartsch-Weth-Willem-05}, it can be shown
  that any minimizer of~$\E_p$ on $\mathcal{M}_p$ is a sign-changing
  critical point of~$\E_p$.  Therefore, least energy nodal solutions
  can be characterized as for the local problem, namely as minimizers
  of the functional on the nodal Nehari set.  This is important from a
  numerical point of view as it gives a natural procedure for seeking
  such solutions.
\end{Rem}

\begin{details}
\begin{Prop}
  If $u\in\mathcal{M}_p$ satisfies $\E_p(u) =
  \min_{v\in\mathcal{M}_p}\E_p(v)$, then $u$ is a critical point
  of $\E_p$ and, so, a least energy nodal solution.
\end{Prop}

\begin{proof}
  Let $c := \min_{v\in\mathcal{M}_p} \E_p(v)$ and $\pi := \{
  t^+ u^+ + t^- u^- : t^+, t^- >0 \}$. Let us assume by contradiction
  that $\norm{\E_p'(u)}_H = \varepsilon >0$. As $u\in
  \mathcal{M}_p$, $u$ changes sign and maximizes $\E_p$ on
  $\pi$. Moreover, $u$ is the unique point on $\pi$ which achieves
  this maximum.

  As $\E_p\in \mathcal{C}^2$, we may construct a ball
  $A\subset \pi$ such that $u\in A$, $0\notin A$, the energy on
  $\partial A$ is strictly less than $c$, the energy on $A$ belongs to
  $\bigl(c-\frac{\varepsilon}{2}, c + \frac{\varepsilon}{2}\bigr)$,
  $\norm{\E_p'(v)}_H>\frac{\varepsilon}{2}$ for any $v\in A$
  and, for any $v\in \partial A$, the vector $(\E'_p(v)[v^+],
  \E_p'(v)[v^-])$ has a direction inside $A$.

  By the traditional deformation lemma, there exists a deformation
  $\Gamma$ such that, for any $v\in A$, there exists $t_v >0$ such
  that, for any $0<t< t_v$, one has
  \begin{enumerate}
  \item $\Gamma(t,v)$ changes sign;
  \item $\E_p(\Gamma(t,v)) < \E_p(v) \le \E_p(u)=c$.
  \end{enumerate}
  By compactness of $A$, we get the existence of $t>0$ such that
  \begin{enumerate}
  \item $\Gamma(t,v)$ changes sign for any $v\in \pi$;
  \item $\Gamma(t,v) = v$ for any $v\in \pi\setminus A$;
  \item $\E_p(\Gamma(t,v)) < \E_p(u)$ for any $v\in A$.
  \end{enumerate}
  So, for any $v\in \pi$, $\E_p(\Gamma(t,v))
  <\E_p(u) =c$.

  To conclude, we then consider
  \begin{equation*}
    \psi : \pi \to \IR\times\IR : v\mapsto
    \bigl(\E_p'(\Gamma(t,v))[\Gamma(t,v)^+],
    \bigl(\E_p'(\Gamma(t,v))[\Gamma(t,v)^-]\bigr).
  \end{equation*}
  As $\Gamma(t,v)=v$ on $\partial A$ and as the vector
  $(\E'_p(v)[v^+], \E_p'(v)[v^-])$ has a direction
  inside $A$, we get by the Miranda's theorem that there exists $w\in
  A$ such that $\psi(w)=0$. This is a contradiction as it would imply
  that $\Gamma(t,w)\in \mathcal{M}_p$ with
  $\E_p(\Gamma(t,w))< c$.
\end{proof}
\end{details}

\section{A priori estimates}
\label{sec:a priori}

\subsection{Equivalence between norms}

In this section, we prove that the norm $\norm{\cdot}$
corresponding to the inner product~\eqref{eq:equiv-inner-prod}
and the traditional norm $\norm{\cdot}_H$ are equivalent.

\begin{Prop}
  The norms $\norm{\cdot}$ and $\norm{\cdot}_H$ are equivalent when
  $V\in L^{\infty}(\Omega)$ and $-\L_K + V$ is positive definite.
\end{Prop}
\begin{proof}
  As $V\in L^{\infty}$ and $H$ embeds continuously in $L^2$, there exists
  a constant $C>0$ such that, for any $u\in H$, one has $\norm{u}^2
  \le (1 + C\abs{V}_{\infty} ) \norm{u}^2_H$.  Moreover, for any
  $\varepsilon\in (0,1)$ and $u\in H$, we have
  \begin{equation*}
    \begin{split}
      \norm{u}^2
      &= \varepsilon \norm{u}^2_H + (1-\varepsilon)\norm{u}^2
      + \varepsilon \int_{\Omega} V(x) u^2 \intd x\\
      &\ge  \varepsilon \norm{u}_{H}^2
      + \bigl(\lambda_1 -\varepsilon \lambda_1 - \varepsilon
      \abs{V}_{\infty}\bigr) \int_{\Omega} u^2.
    \end{split}
  \end{equation*}
  where $\lambda_1 > 0$ since the operator $-\L_K
  + V$ is positive definite. Taking $\varepsilon$ small, we conclude
  the proof.
\end{proof}

Thus, for this new norm $\norm{\cdot}$ on $H$, we can use
Poincar\'e's and Sobolev's inequalities. In the following, we assume
that we work with $H$ endowed with the norm $\norm{\cdot}$ and the
inner product~\eqref{eq:equiv-inner-prod}.

\subsection{Upper bound}

In this section, let us consider $(u_p)_{2<p<2^*_s}$ a family of
ground state solutions (resp.\ least energy nodal solutions) for the problem
 \begin{equation}
\label{eq:pbmPlambda}
     \begin{cases}
      - \L_K u(x) + V(x)u(x) = \lambda \abs{u(x)}^{p-2}u(x),
      &\text{for }
      x\in\Omega,\\
      u(x)=0, &\text{for } x\in \IR^N \setminus \Omega,
    \end{cases}
  \end{equation}
where $\lambda = \lambda_1$ (resp.\ $\lambda_2$),
then the $v_p:=\lambda^{1/(p-2)}u_p$ are solutions to
Problem~\eqref{eq:pbmP}.
Let us note $\tilde\E_p$ the functional associated
to~\eqref{eq:pbmPlambda}:
\begin{equation*}
  \tilde\E_p(u)= \frac{1}{2} \norm{u}^2_{H} +
  \frac{1}{2}\int_\Omega V(x)u^2 \intd x
  - \frac{\lambda}{p} \abs{u}_{p}^p
  = \frac{1}{2} \norm{u}^2 - \frac{\lambda}{p} \abs{u}_{p}^p,
\end{equation*}
and let $\tilde{\mathcal N}_p$ (resp. $\tilde{\mathcal M}_p$) be
its corresponding Nehari manifold (resp.\ nodal Nehari set).
As the symmetries of $u_p$ and $v_p$ are the same and $\E_p(v_p) =
\lambda^{2/(p-2)} \tilde\E_p(u_p)$, it suffices to study the
ground state and least energy nodal solutions
to~\eqref{eq:pbmPlambda}.

\begin{Lem}
  \label{dimE1=1}
  Assume $-\L_K + V$ is positive definite.  All eigenfunctions
  in $E_1 \setminus\{0\}$ are nonnegative or nonpositive.
  Thus $\dim E_1 = 1$.
  All eigenfunctions of $E_2 \setminus\{0\}$ change sign.
\end{Lem}
\begin{proof}
  Since $-\L_K + V$ is positive definite, $\norm{\cdot}$ is a norm.
  Suppose on the contrary that there exists $u \in E_1\setminus\{0\}$
  with both $u^+ \ne 0$ and $u^- \ne 0$.  Then, one has
  \begin{align*}
    \frac{\norm{u}^2}{\abs{u}_2^2}
    = \frac{\norm{u^+ + u^-}^2}{
      \abs{u^+ + u^-}_2^2}
    &= \frac{\norm{u^+}^2 + 2 \langle u^+, u^- \rangle_H
      +  \norm{u^-}^2}{
      \abs{u^+}^2_2 + \abs{u^-}^2_2}\\[1\jot]
    &> \frac{\norm{u^+}^2 + \norm{u^-}^2}{\abs{u^+}^2_2 + \abs{u^-}^2_2}
    \ge \min\Bigl\{ \frac{\norm{u^+}^2}{ \abs{u^+}^2_2},\
    \frac{\norm{u^-}^2}{ \abs{u^-}^2_2} \Bigr\}  .
  \end{align*}
  which contradicts the variational characterization of~$\lambda_1$.
  Because the cone $K := \{ u : u \ge 0\}$ is closed and pointed,
  it is standard to show that the fact that $E_1$ is made only
  of elements of $K$ or $-K$ implies
  $\dim E_1 \le 1$.

  Finally, let $\phi_2 \in E_2\setminus\{0\}$ and
  suppose on the contrary that $\phi_2 \ge 0$ (the case $\phi_2 \le 0$
  is similar).
  Because $\phi_2 \perp E_1$ in $L^2(\Omega)$, one concludes that
  $\phi_2 = 0$ a.e.\ on $\{\phi_1 > 0\}$.
  Thus, $0 = \langle \phi_1, \phi_2 \rangle
  = \langle \phi_1, \phi_2 \rangle_H
  = -2 \int_{\IR^N \times\IR^N} \phi_1(x) \phi_2(y) K(x-y) \intd(x,y)
  < 0$,
  a contradiction.
\end{proof}

\begin{Prop}
  \label{bound}
  The family $(u_p)_{2<p< \bar p}$ is bounded in $H$ for the norm
  $\norm{\cdot}$, for any $\bar p < 2^*_s$.
\end{Prop}

\begin{proof}
  Let us start with ground state solutions.
  Consider $\phi_1\in E_1$ such that $\norm{\phi_1}=1$.  If
  \begin{equation*}
    t_p := \bigg( \frac{1}{\lambda_1
      \abs{\phi_1}_{p}^p}\biggr)^{1/(p-2)} >0,
  \end{equation*}
  then $t_p\phi_1\in \mathcal{N}_p$ i.e., $t_p^2 \norm{\phi_1}^2 =
  t_p^p \lambda_1 \abs{\phi_1}_{p}^p$.  We shall prove that
  $p\mapsto t_p : (2, \bar p] \to \IR$ is bounded.
  By continuity, it is enough to
  check that $t_p$ converges to some $t_*<+\infty$ as $p\to 2$.  We
  have
  \begin{equation*}
    \lim_{p\to 2} \ln  t_p
    = -\lim_{p\to 2} \frac{\ln (\lambda_1 \abs{\phi_1}_{p}^p) }{p-2}.
  \end{equation*}
  Since $\lambda_1\abs{\phi_1}_{2}^2 =1$, we can use
  L'Hospital's rule.  Remark that $\partial_p \int_{\Omega}
  \abs{\phi_1}^p = \int_\Omega \ln \abs{\phi_1} \, \abs{\phi_1}^p$
  by Lebesgue's dominated convergence theorem.  Then, for $p\to 2$,
  \begin{equation*}
    \lim_{p\to 2} t_p
    = \exp\biggl(-\frac{\int_\Omega \ln\abs{\phi_1} \,
      \abs{\phi_1}^2}{\int_{\Omega} \abs{\phi_1}^2} \biggr)
    < +\infty.
  \end{equation*}
  Since $u_p \in \tilde{\mathcal N}_p$ has the lowest energy,
  $\bigl(\frac{1}{2} - \frac{1}{p} \bigr)
  \norm{u_p}^2 = \tilde\E_p(u_p) \le \tilde\E_p (t_p\phi_1) =
  \bigl(\frac{1}{2} - \frac{1}{p}\bigr) t_p^2$ concluding this case.

  Let us now treat the case of least energy nodal solutions.  Pick
  $\phi_2\in E_2\setminus\{0\}$ and let $t^+_p > 0$ and $t^-_p > 0$ be
  such that $t_p^+ \phi_2^+ + t_p^- \phi_2^- \in \tilde{\mathcal M}_p$
  (they exist because $\phi_2$ changes sign, see
  page~\pageref{proj-nodal-Nehari}).  Expanding the equations
  $\tilde\E_p'(t_p^+ \phi_2^+ + t_p^- \phi_2^-)[\phi_2^\pm] = 0$
  yields
  \begin{align}
    \label{eq:t+}
    &t^+_p \norm{\phi_2^+}^2
    + t^-_p \langle \phi_2^+, \phi_2^-\rangle
    - \lambda_2 (t^+_p)^{p-1} \abs{\phi_2^+}_p^p = 0\\
    \label{eq:t-}
    &t^-_p \norm{\phi_2^-}^2
    + t^+_p \langle \phi_2^+, \phi_2^-\rangle
    - \lambda_2 (t^-_p)^{p-1} \abs{\phi_2^-}_p^p = 0
  \end{align}
  The fact that $\phi_2$ is a second eigenfunction reads
  $\langle \phi_2, w \rangle = \lambda_2 \int_\Omega \phi_2 w$
  for all $w \in H$.  In particular, taking $w$ as $\phi_2^+$ and
  $\phi_2^-$ yields
  \begin{equation}
    \label{eq:eigenfunctions+-}
    \norm{\phi_2^+}^2 = - \langle \phi_2^+, \phi_2^-\rangle
    + \lambda_2 \abs{\phi_2^+}_2^2
    \quad\text{and}\quad
    \norm{\phi_2^-}^2 = - \langle \phi_2^+, \phi_2^-\rangle
    + \lambda_2 \abs{\phi_2^-}_2^2 .
  \end{equation}
  Substituting back in \eqref{eq:t+}--\eqref{eq:t-}, one deduces that
  \begin{equation*}
    t^+_p \bigl(
    \abs{\phi_2^+}_2^2 - (t_p^+)^{p-2} \abs{\phi_2^+}_p^p \bigr)
    = - t^-_p \bigl(
    \abs{\phi_2^-}_2^2 - (t_p^-)^{p-2} \abs{\phi_2^-}_p^p \bigr).
  \end{equation*}
  Thus $\abs{\phi_2^+}_2^2 - (t_p^+)^{p-2} \abs{\phi_2^+}_p^p$
  and $\abs{\phi_2^-}_2^2 - (t_p^-)^{p-2} \abs{\phi_2^-}_p^p$ always
  have opposite signs.
  Let us show that $t^+_p$ and $t^-_p$ are bounded as $p \to 2$.
  Let us split these families into (possibly) two sub-families
  according to the sign of
  $\abs{\phi_2^+}_2^2 - (t_p^+)^{p-2} \abs{\phi_2^+}_p^p$.
  We deal with the subfamily for which
  $\abs{\phi_2^+}_2^2 - (t_p^+)^{p-2} \abs{\phi_2^+}_p^p \ge 0$
  (for the other one, this expression is $< 0$, so
  $\abs{\phi_2^-}_2^2 - (t_p^-)^{p-2} \abs{\phi_2^-}_p^p > 0$ and
  the argument is similar).  This inequality can be rewritten as
  \begin{equation}
    \label{eq:bound:t+}
    t^+_p \le \left( \frac{\abs{\phi_2^+}_2^2}{
        \abs{\phi_2^+}_p^p} \right)^{1/(p-2)}
    \xrightarrow[p\to 2]{}
    \exp\biggl( -
      \frac{\int_\Omega \ln\abs{\phi_2^+} \, \abs{\phi_2^+}^2}{
        \abs{\phi_2^+}_2^2 } \biggr)
  \end{equation}
  (with $s \ln\abs{s}$ understood as $0$ when $s=0$)
  where the convergence results from arguments similar to those used
  in the ground state case.  Thus $t^+_p$ is bounded for $p$ close to
  $2$ and equation~\eqref{eq:t+} implies that the same holds for
  $t^-_p$.

  The conclusion follows
  easily since
  \begin{equation*}
    \Bigl(\frac{1}{2}-\frac{1}{p}\Bigr) \norm{u_p}^2
    =  \tilde\E_{p}(u_{p})
    \le \tilde\E_{p}(\hat{v}_{p})
    = \Bigl(\frac{1}{2}-\frac{1}{p}\Bigr)\norm{\hat{v}_p}^2
  \end{equation*}
  and $\norm{\hat{v}_p}$ is bounded for $p$ close to~$2$ because
  $t^+_p$ and $t^-_p$ are and $v_p \to u_*$.
\end{proof}

We may thus assume that $u_p$ weakly converges, up to a subsequence, to
some $u_*\in H$ as $p\to 2$.

\pagebreak[3]
\begin{Prop}
  \label{weak}
  If $(u_p)$ converges weakly in $H$ to
  $u_*$ as $p\to 2$ then $u_* \in E_1$ (resp.\ $E_2$).
\end{Prop}

\begin{proof}
  For every $v\in H$, one has
  \begin{equation*}
    0 =
    \tilde\E'_p (u_p)[v] = \langle u_p, v  \rangle - \lambda
    \int_{\Omega} \abs{u_p}^{p-2}u_p v
  \end{equation*}
  where $\lambda = \lambda_1$ (resp.\ $\lambda = \lambda_2$) for
  ground state solutions (resp.\ least energy nodal solutions).  Since
  $u_p$ converges weakly to $u_*$, the first term converges to $\langle
  u_*,v \rangle$. Moreover, $u_p \to u_*$ in $L^q(\Omega)$ for $1\le
  q< 2^*_s$.  So, up to a subsequence, $u_p\to u_*$ a.e.\ and
  there is $f\in L^{2}(\Omega)$
  such that $\abs{u_p}\le f$ almost everywhere.  By Lebesgue's
  dominated convergence theorem, the
  second term converges to $\int_{\Omega} u_* v$ as $\bigabs{
    \abs{u_p}^{p-2} u_p v} \le \abs{ \max\{f,1\}}^{\bar p - 1} \, \abs{v} \in
  L^1(\Omega)$
  when $p \le \bar p < 2^*_s$.
  As the limit does not depend on the subsequence, the whole sequence
  converges.
  Thus, $u_*$ is a weak solution to $-\L_K u +
  V(x) u = \lambda u$.
\end{proof}

\subsection{Lower bound}

\begin{Prop}
  \label{away}
  If $(u_p)$ converges weakly to $u_*$ in $H$ as $p\to 2$ then $u_*
  \ne 0$.
\end{Prop}

\begin{proof}
  We first treat the case when $u_p$ is a ground state solution.  By
  H\"older's inequality, we have $\abs{u_p}_{p}^2 \le
  \abs{u_p}_{2}^{2(1-\omega)} \abs{u_p}_{2^*_s}^{2\omega}$ with
  $\omega = \frac{2^*_s}{2^*_s-2}\frac{p-2}{p}$.  Then, by using
  Poincar\'e and Sobolev inequalities and since $u_p$ belongs to the
  Nehari manifold, we have
  \begin{equation*}
    \abs{u_p}_{p}^2  \le \bigl( \lambda_1^{-1} \norm{u_p}^2
    \bigr)^{1-\omega} \bigl( S^{-1} \norm{u_p}^2\bigr)^{\omega}
    = \bigl(\abs{u_p}_{p}^p\bigr)^{1-\omega} \,\abs{u_p}_{p}^{p\omega}
    (S^{-1}\lambda_1)^{\omega}
    = \abs{u_p}_{p}^p  \, (S^{-1}\lambda_1)^{\omega}.
  \end{equation*}
  Thus, $\abs{u_p}_{p}^p\ge (S
  \lambda_{1}^{-1})^{2^*_s/(2^*_s-2)}$.  Using the compact embeddings
  and Lebesgue's dominated convergence theorem, one has
  $\abs{u_*}_{2}^2 = \lim_{p\to 2}\abs{u_p}_{p}^p >0 $.

  In the case of least energy nodal
  solutions, we claim that there exists $v_p = t_{p}^+ u_{p}^+ +
  t_{p}^- u_p^- \in \mathcal{N}_p \cap E_1^{\perp}$ such that
  $\norm{v_p}\le \norm{u_p}$. Then, by the same argument as for the
  ground state case (with $\lambda_2$ instead of $\lambda_1$ because
  $v_p \perp E_1$), we get that $v_p$ stays away from zero which is
  enough to conclude.  To prove the claim, consider the line segment
  \begin{equation*}
    T:[0,1] \to H\setminus\{0\}:
    \alpha \mapsto (1-\alpha)u_p^{+}+\alpha u_p^{-}.
  \end{equation*}
  For all
  $\alpha\in [0,1]$, there exists a unique $t_{\alpha}>0$ such
  that $t_{\alpha}T(\alpha) \in {\mathcal{N}}_{p}$.
  This $t_{\alpha}$ can be written explicitly and is easily seen to
  be continuous w.r.t.\ $\alpha$.
  For $\alpha = 0$, we have
  $\int_{\Omega}t_{\alpha}u^{+}_p \phi_{1} > 0$ and, for $\alpha = 1$, we
  have $\int_{\Omega} t_{\alpha}u^{-}_p \phi_{1} < 0$.  So, by continuity,
  there is a $\alpha^{*} \in (0,1)$ such that
  $\int_{\Omega}t_{\alpha^{*}} T(\alpha^{*}) \phi_{1} = 0$ and
  $t_{\alpha^{*}} T(\alpha^{*}) \in {\mathcal{N}}_{p}$.  We just set $t_p^{+} :=
  t_{\alpha^{*}} (1-\alpha^{*})$ and $t_p^{-} := t_{\alpha^{*}}
  \alpha^{*}$ to conclude.  By definition of
  $\mathcal{M}_p \subseteq \mathcal{N}_p$, we
  get that $\tilde\E_p(v_p) \le \tilde\E_p(u_p)$ and
  so, $\norm{v_p}\le \norm{u_p}$.
\end{proof}

\begin{Prop}
  \label{convH}
  Ground state solutions (resp.\ least energy nodal solutions)
  to~\eqref{eq:pbmPlambda}
  converge, up to a subsequence, in $H$ to
  some $\phi_1^*\in E_1 \setminus\{0\}$ (resp.\
  $\phi_2^* \in E_2 \setminus\{0\}$).
\end{Prop}

\begin{proof}
 Let $(u_p)$ be a family of ground state solutions
  of~\eqref{eq:pbmPlambda}.  The argument is identical for least energy
  nodal solutions.  By Proposition~\ref{bound}, for any sequence
  $p_n \to 2$, there exists a subsequence, still denoted
  $p_n$, such that $u_{p_n}$ converges weakly in $H$ to some $u^*
  \in H$.
  Proposition~\ref{weak} and~\ref{away} imply that $u_* \in E_1
  \setminus\{0\}$.
  Finally, the compact embedding of $H$ into $L^q$ for $1\le q < 2^*_s$ and
  \begin{align*}
    0 &= \tilde\E_{p_n}'(u_{p_n}) [u_{p_n} - u^*]
    - \tilde\E_2'(u^*) [u_{p_n} - u^*]\\
    &= \norm{u_{p_n} - u^*}^2
    \begin{aligned}[t]
      &-  \frac{\lambda_1}{p_n}
      \int_\Omega \abs{u_{p_n}}^{p_n-2} u_{p_n}(u_{p_n} - u^*)
      +  \frac{\lambda_1}{2} \int_\Omega u^* (u_{p_n} - u^*)
    \end{aligned}
  \end{align*}
  show that $u_{p_n} \to u_*$ in $H$.
\end{proof}

Remark that, from propositions~\ref{bound},~\ref{weak}
and~\ref{away}, we get the first conclusion of
Theorem~\ref{Thm-intro:simple}.

\section{Symmetries and uniqueness via implicit function
  theorem}
\label{sec:IFT}

In this section, we prove
the uniqueness (up to its sign) in $H$ of a ground state
solution (resp.\ least energy nodal solution) to
Problem~\eqref{eq:pbmPlambda} when $\dim E_1 = 1$ (resp.\ $\dim E_2 = 1$).
To start, we consider the following family of problems parametrized by
$2<p<2^*_s$ and  $\lambda \in\IR$:
\begin{equation}
  \label{eq:pbmIFT}
  \begin{cases}
    (- \L_K + V)u = \lambda \abs{u}^{p-2}u,
    &\text{in }
    \Omega,\\
    u=0, &\text{in } \IR^N \setminus \Omega,\\
    \norm{u}=1. &
  \end{cases}
\end{equation}

\begin{Prop}
  \label{propIFT}
  When $\dim E_1 =1$ (resp.\ $\dim E_2=1$), there exists a unique curve
  of solutions $p \mapsto (p,u_p^*,\lambda_p)$
  solving~\eqref{eq:pbmIFT} starting from $(2,\phi_1,
  \lambda_1)$ (resp.\ $(2,\phi_2, \lambda_2)$)
  where $\phi_1 \in E_1$ with $\norm{\phi_1} = 1$
  (resp.\ $\phi_2 \in E_2$ with $\norm{\phi_2} = 1$).
  There is also a unique curve of solutions starting
  from $(2, -\phi_1, \lambda_1)$
  (resp.\ $(2, -\phi_2, \lambda_2)$) which is given by
  $p \mapsto (p, -u_p^*,\lambda_p)$.
\end{Prop}

\begin{proof}
  We make the proof for the ground states, the other case being
  similar.  Let $\psi$ be the function
  \begin{equation*}
    \psi : (2,2^*_s) \times H \times \IR \to H\times \IR : (p,u,\lambda)
    \mapsto \bigl(u - \lambda (-\L_K + V)^{-1}(\abs{u}^{p-2}u),\
    \norm{u}^2-1\bigr),
  \end{equation*}
  so that $(p, u, \lambda)$ is a root of $\psi$ if and only if $u$
  is a solution to~\eqref{eq:pbmIFT}.  To pursue our
  goal, we shall use the implicit function theorem as well as the
  closed graph theorem. First, we have to show that the Fr\'echet
  derivative of $\psi$ at $(2,\phi_1,\lambda_1)$ with respect to
  $(u,\lambda)$ is bijective on $H\times\IR$. Let us remark that
  \begin{equation}
    \label{eq:uaun}
    \partial_{(u,\lambda)}\psi(2,\phi_1,\lambda_1)[(v,t)]
    = \bigl(v - \lambda_1
    (-\L_K + V)^{-1} v - t (-\L_K + V)^{-1}\phi_1,\
    2 \langle \phi_1, v\rangle\bigr).
  \end{equation}
  For injectivity, let us start by showing that
  $\partial_{(u,\lambda)}\psi(2,\phi_{1},\lambda_{1})[(v,t)] = 0$ if and
  only if
  \begin{equation}
    \label{eq:1to1}
    \begin{cases}
      v - \lambda_{1}(-\L_K + V)^{-1}v =0,\\
      t=0,\\
      v \text{ is orthogonal to } \phi_{1} \text{ in } H
      \text{.}
    \end{cases}
  \end{equation}
  Clearly, \eqref{eq:1to1} is sufficient.  For its necessity,
  observe that the second component of~\eqref{eq:uaun}
  implies that $\phi_{1}$ is orthogonal to $v$ in $H$ and
  thus also in $L^2(\Omega)$ because $\phi_1$ is an eigenfunction.
  Taking the $L^2$-inner product of the first component
  of~\eqref{eq:uaun} with $\phi_1$ yields $t=0$, hence
  completing the equivalence.  Now, the only solution
  of~\eqref{eq:1to1} is $(v,t)=(0,0)$
  because the first equation and the dimention~$1$ of $E_1$  imply $v=
  \alpha \phi_1$ for some $\alpha\in \IR$ and then the third property
  implies $v=0$.  This
  concludes the proof of the injectivity.
  Let us now show that, for any $(w,s) \in H \times \IR$, the
  equation $\partial_{(u,\lambda)}\psi(2,\phi_{1},\lambda_{1})[(v,t)] =
  (w,s)$ always possesses at least one solution $(v,t) \in H
  \times \IR$.  One can write $w = \bar w \phi_1 + \widetilde w$ for some
  $\bar w \in \IR$ and $\widetilde w \in H$ orthogonal to $\phi_1$ in $H$.
  Similarly, one can decompose $v = \bar v \phi_1 + \widetilde v$.
  Arguing as for the first part, the equation can be written
  \begin{equation}
    \label{eq:onto}
    \begin{cases}
      \widetilde v - \lambda_{1}(-\L_K + V)^{-1} \widetilde v =
      \widetilde w,\\
      t = -\lambda_1 \bar w,\\
      \bar v = s/2  \text{.}
    \end{cases}
  \end{equation}
  The existence of the solution
  $\widetilde v$ results from the Fredholm alternative.
  This
  concludes the proof that
  $\partial_{(u,\lambda)}\psi(2,\phi_{1},\lambda_{1})$ is onto and thus of
  the existence and uniqueness of the branch $p \mapsto
  (p,u_{p}^*,\lambda_p)$
  emanating from $(2,\phi_{1},  \lambda_{1})$.
  It is clear that $p \mapsto (p, -u_{p}^*, \lambda_p)$ is a branch
  emanating from $(2, -\phi_{1}, \lambda_{1})$ and, using as above the
  implicit function theorem at that point, we know it is the only one.
\end{proof}

\begin{Thm}
  \label{thm:main}
  Assume $\dim E_1=1$ (resp.\ $\dim E_2=1$).
  For $p$ close to $2$,
  ground state solutions (resp.\ least energy nodal solutions)
  to~\eqref{eq:pbmPlambda} are
  unique (up to their sign) and
  possess the same symmetries as $\phi_1$ (resp.\ $\phi_2$).
\end{Thm}

\begin{proof}
  We make the argument for the ground state solutions as it is
  identical for the other case.
  Let $(u_p)_{2<p<2^*_s}$ be a family of ground state solutions
  to
  Problem~\eqref{eq:pbmPlambda}
  and $p_n \to 2$.  It suffices to show that, up to a subsequence,
  $(u_{p_n})$ possess the same symmetries as $\phi_1$.
  Thanks to Proposition~\ref{convH}, we can assume
  without loss of generality
  that $u_{p_n} \to u_*\in
  E_1\setminus\{0\}$.
  Thus $u_* = \alpha \phi_1$ for some $\alpha\ne 0$.
  Notice
  that $u$ is a solution to~\eqref{eq:pbmPlambda} if
  and only if ${u}/{\norm{u}}$ is a solution to~\eqref{eq:pbmIFT}
  with $\lambda = \lambda_1 \norm{u}^{p-2}$.
  Also, since the family $(u_{p})$ remains bounded
  away from~$0$, one has $u_{p_n} \norm{u_{p_n}}^{-1}\to
  \sign(\alpha) \phi_1$
  and $\lambda_1 \norm{u_{p_n}}^{p_n-2}\to \lambda_1$.
  Then, for $n$ large,
  Proposition~\ref{propIFT} implies
  $u_{p_n} \norm{u_{p_n}}^{-1} = \sign(\alpha) \, u^*_{p_n}$
  Hence $u_{p_n}$ is unique up to its sign.
  Also, $u_{p_n}$ respects the \text{(anti-)symmetries} of $\phi_1$.
  Indeed, let us
  consider a direction $d$ such that $\phi_1$ is symmetric (resp.\
  anti-symmetric) with respect to $d$.   If $u_{p_n}$ is
  not, let us consider $u'_{p_n}$ the symmetric (resp.\
  anti-symmetric) image
  of $u_{p_n}$.  Because $\phi_1$ is
  symmetric (resp.\
  anti-symmetric) in the direction~$d$, $u'_{p_n} \to
  \alpha \phi_1$ (resp.\ $u'_{p_n} \to -\alpha \phi_1$).   Arguing as
  before, we conclude that
   $u_{p} \norm{u_{p}}^{-1} = \pm\text{sign}(\alpha) \, u^*_{p}
    = \pm  u'_{p} \norm{u'_{p}}^{-1}$,
  which concludes the proof.
\end{proof}

This directly gives the second conclusion of
Theorem~\ref{Thm-intro:simple} and thus completes it proof.

\section{Asymptotic symmetries : Lyapunov-type reduction}
\label{sec:laypunov}

In this section, we present an abstract symmetry result which is
useful when  $\dim E_1\ne 1$ or $\dim E_2 \ne 1$.  By
Propositions~\ref{bound} and~\ref{away}, it will
give the proof of Theorem~\ref{Thm-intro:simple3}.  The idea is to
show that, for $p$ close to $2$, a priori
bounded solutions of \eqref{eq:pbmP} can be distinguished by their
projections on the  eigenspaces $E_i$. This will follow from
Proposition \ref{Prop:uniqueness} below.

\begin{Lem}\label{Lem:uniq}
  Let $i\ge 1$. There exists $\varepsilon>0$ such that if
  $a \in L^{N/(2s)}(\Omega)$ satisfies
  $\abs{a-\lambda_{i}}_{N/(2s)}<\varepsilon$ and $u$ solves
  \[
  \left\{
    \begin{aligned}
      -\L_K u + V u &= a(x)u, & &\text{in $\Omega$,}\\
      u&=0, & &\text{in $\IR^N \setminus  \Omega$,}
    \end{aligned}
  \right.
  \]
  then $P_{E_{i}}u = 0 \limplies u=0$
  where $P_{E_{i}}$ is the orthogonal projector on~$E_i$.
\end{Lem}
\begin{proof}
  Assume by contradiction that
  there exists a nontrivial solution $u$ such that $P_{E_{i}}u =
  0$. Let $w=P_{E_{1}\oplus \cdots\oplus E_{i-1}}u$
  (with $w=0$ if $i=1$ so one does not need~\eqref{eq:norm-w}) and
  $z=P_{(E_1\oplus \cdots\oplus E_{i})^\perp} u$. Taking
  successively $w$ and $z$ as test functions and using Poincar\'e,
  Sobolev and H\"older inequalities, we infer that
  \begin{align*}
    \norm{w}^2
    &= \lambda_{i}\abs{w}_{2}^{2}
    + \int_{\Omega}(a(x)-\lambda_i)uw \intd x
    \ge \frac{\lambda_i}{\lambda_{i-1}} \norm{w}^2
    - C\abs{a(x)-\lambda_i}_{\frac{N}{2s}} \norm{w} \norm{u},\\
    \norm{z}^2
    &= \lambda_{i}\abs{z}_{2}^{2}
    + \int_{\Omega}(a(x)-\lambda_i)uz \intd x
    \le \frac{\lambda_i}{\lambda_{i+1}} \norm{z}^2
    + C \abs{a(x)-\lambda_i}_{\frac{N}{2s}} \norm{ z} \norm{ u}.
  \end{align*}
  We deduce that
  \begin{align}
    \label{eq:norm-w}
      \norm{w}
      &\le \frac{\lambda_{i-1}C}{\lambda_{i}-\lambda_{i-1}}
    \abs{a-\lambda_i}_{\frac{N}{2s}}\norm{ u},\\
    \label{eq:norm-z}
      \norm{z}
      &\le \frac{\lambda_{i+1}C}{\lambda_{i+1}-\lambda_i}
    \abs{a-\lambda_i}_{\frac{N}{2s}}\norm{u}.
  \end{align}
  Since $\norm{u}^2 = \norm{w}^2 + \norm{z}^2$, we get a contradiction
  when $ \abs{a-\lambda_i}_{{N}/({2s})}$ is small enough
  for the coefficients of $\norm{u}$ in
  \eqref{eq:norm-w}--\eqref{eq:norm-z} to be less than~$1$.
\end{proof}

The next result must be compared with the use of the implicit function
theorem in the previous section.  Note that, this time,
uniqueness is not guaranteed.

\begin{Prop}\label{Prop:uniqueness}
  Let $i \ge 1$.
  Let $(u_p)_{2 < p < 2^*_s}$ and $(v_p)_{2 < p < 2^*_s}$ be two
  families of solutions to
  \[
  \left\{
    \begin{aligned}
      -\L_K u + V u &= \lambda_i \abs{u}^{p-2}u,
      & &\text{in $\Omega$,}\\
      u&=0, & &\text{in $\IR^N \setminus  \Omega$,}
    \end{aligned}
  \right.
  \]
  Let $p_n \to 2$ be
  such that $u_{p_n} \rightharpoonup \phi_i$
  for some $\phi_i \in E_i\setminus\{0\}$,
  $(v_{p_n})$ is bounded in~$H$, and
  the Lebesgue measure of
  the zero set of $\phi_i$, namely $\{ x \in \Omega : \phi_i(x) = 0 \}$,
  is zero.
  If, for $n$ large, $P_{E_{i}}u_{p_n} = P_{E_{i}}v_{p_n}$, then,
  for all $n$ large enough, $u_{p_n} = v_{p_n}$.
\end{Prop}

\begin{proof}
  Suppose on the contrary that there is a subsequence,
  still denoted $(p_n)$, such that,
  for all $n$, $P_{E_{i}}u_{p_n} = P_{E_{i}}v_{p_n}$
  and $u_{p_n} \ne v_{p_n}$.
  Since $(v_{p_n})$ is bounded in $H$, up
  to a subsequence, $v_{p_n} \rightharpoonup v_*$ in $H$ for some $v_*
  \in H$.  Clearly $\phi_i = P_{E_i} \phi_i = P_{E_i} v_*$.
  Compact embeddings of $H$ imply that $v_{p_n} \to v_*$ in
  $L^q(\Omega)$ for every $q \in [1, 2^*_s)$ and thus $v_* \in E_i$.
  Therefore $v_* = \phi_i$.
  Observe that
  \begin{equation}
    \label{eqDiff}
    \left\{
      \begin{aligned}
        (-\L_K + V) (u_p-v_p)&=a_p(x)(u_p-v_p),
        & & \text{in $\Omega$},\\
        u_p-v_p & = 0, & & \text{in $\IR^N\setminus \Omega$},
      \end{aligned}
    \right.
  \end{equation}
  where
  \begin{equation*}
    a_p(x)
    := (p-1) \int_0^1
    \bigabs{v_p(x) + \theta(u_p(x) - v_p(x))}^{p-2} \intd\theta .
  \end{equation*}
  It is readily seen that $a_{p_n}(x) \to \lambda_i$ for a.e.\ $x$
  such that $\phi_i(x) \ne 0$.
  Noting that $\bigabs{v_p(x) + \theta(u_p(x) - v_p(x))}^{p-2}
  \le \abs{v_p(x)}^{p-2} + \abs{u_p(x)}^{p-2}$,
  we can apply Lebesgue's dominated convergence
  theorem to deduce that $a_{p_n} \to \lambda_i$
  in $L^{N/(2s)}(\Omega \setminus \{\phi_i = 0\})$.
  Since $\{\phi_i = 0\}$ has zero measure, this convergence also holds
  in $L^{N/(2s)}(\Omega)$.

  In particular, for $n$ large enough,
  $\abs{a_{p}-\lambda_i}_{{N}/({2s})} < \varepsilon$
  where $\varepsilon > 0$ is given by Lemma~\ref{Lem:uniq}.
  Since $P_{E_{i}}(u_{p_n}-v_{p_n}) = 0$, Lemma \ref{Lem:uniq} implies
  $u_{p_n} = v_{p_n}$.  This contradiction concludes the proof.
\end{proof}

\begin{Rem}
In the nonlocal setting, the unique continuation property is a difficult subject and it
has only recently been investigated in \cite{felli}. In particular, by \cite[Theorem 1.4]{felli},  
the Lebesgue measure of $\{ x \in \Omega : \phi_i(x) = 0 \}$ is indeed equal to zero for the model operator $-\L_K=(-\Delta)^s$.
\end{Rem}

\begin{Thm}\label{Thm:abssym}
  Let $(u_p)_{2 < p < 2^*_s}$ be a family of ground state (resp.\
  least energy nodal) solutions to Problem~\eqref{eq:pbmPlambda}
  and let $i = 1$ (resp.\ $i = 2$).
  Let $G$ be a group acting on $H$ in such a
  way that there exists $C > 0$ so that,
  for every $g\in G$, $u \in H$, and $p$ close to~$2$,
  \begin{equation*}
    \text{(i)}\hspace{0.7em} g(E_i)=E_i,\qquad
    \text{(ii)}\hspace{0.7em} g(E_i^{\perp})=E_i^{\perp}, \qquad
    \text{(iii)}\hspace{0.7em} \E_{p}(gu)=\E_{p}(u), \qquad
    \text{(iv)}\hspace{0.7em} \norm{gu} \le C \norm{u}.
  \end{equation*}
  Assume the zero set of any functions in $E_i\setminus\{0\}$
  has zero Lebesgue measure.
  Then, for $p$ close enough to~$2$,
  $u_{p}$ is invariant under the isotropy group
  $G_{\alpha_{p}} = \{ g \in G : g \alpha_p = \alpha_p\}$
  of $\alpha_{p}:=P_{E_i}u_{p}$.
\end{Thm}
\begin{proof}
  Suppose on the contrary that there exists sequences $p_n \to 2$
  and $g_n \in G_{\alpha_{p_n}}$, where $\alpha_{p_n} := P_{E_i}u_{p_n}$,
  such that $g_n u_{p_n} \ne u_{p_n}$ for all~$n$.
  According to Proposition~\ref{convH}, one can assume w.l.o.g.\ that
  $u_{p_n} \to \phi^*_i \in E_i\setminus\{0\}$.

  It follows from~(iii) that, for all $v \in H$,
  $\E'_{p_n}(g u_{p_n})[v] = \E'_{p_n}(u_{p_n})[g^{-1} v]$,
  so $g_n u_{p_n}$ are also solutions to Problem~\eqref{eq:pbmPlambda}.
  Moreover, given that
  \begin{equation*}
    g u_{p} = g(P_{E_i}u_p) + g(P_{E_i^\perp} u_p)
    \qquad\text{with }
    g(P_{E_i}u_p) \in E_i  \text{ and }
    g(P_{E_i^\perp} u_p) \in E_i^\perp,
  \end{equation*}
  one deduces that $P_{E_i}(g u_p) = g(P_{E_i}u_p)$.
  In particular, $P_{E_i}(g_n u_{p_n}) = g_n \alpha_{p_n} = \alpha_{p_n}$.
  As a consequence,
  $P_{E_i}(g_n u_{p_n}) = P_{E_i}(u_{p_n})$.
  Moreover, property (iv) implies that $(g_n u_{p_n})$ is bounded
  in~$H$.
  Proposition~\ref{Prop:uniqueness} thus implies that
  $g_n u_{p_n} = u_{p_n}$ for $n$ large which contradicts our initial
  negation of the thesis.
\end{proof}

\begin{Thm}[Localization of limit functions]
  \label{thm:localization}
  Let $(u_p)_{p > 2}$ be a family of least energy nodal solutions
  to Problem~\eqref{eq:pbmPlambda}.  Let $p_n \to 2$ be such that
  $u_{p_n} \to u_*$ in~$H$.
  Then $u_* \in E_2\setminus\{0\}$ and it achieves
  the minimum of the reduced functional
  \begin{equation*}
    \E_* : E_2 \to \IR : u \mapsto
    \tfrac{1}{2} \int_\Omega u^2 - u^2 \ln u^2
  \end{equation*}
  subject to the constraint $u \in \mathcal{N}_*$ where
  $\mathcal{N}_*$ is the reduced Nehari manifold
  \begin{equation*}
    \mathcal{N}_* := \bigl\{ u \in E_2\setminus\{0\} :
    \E'_*(u)[u] = 0 \bigr\}.
  \end{equation*}
  In particular, $u_*$ satisfies
  \begin{equation}
    \label{eq:limit-equation}
    \begin{cases}
      (-\L_K + V) u_* = \lambda_2 u_*& \text{in } \Omega,\\[1\jot]
      u_* = 0&  \text{in } \IR^N\setminus\Omega,\\[1\jot]
      \displaystyle
      \int_\Omega u_* \ln\abs{u_*} \, v = 0&
      \text{for all } v \in E_2.
    \end{cases}
  \end{equation}
\end{Thm}

\begin{Rem}\ 
  \begin{enumerate}
  \item Quantities like $s \ln s$ are understood as being $0$ when $s
    = 0$.
  \item For all $v \in E_2\setminus\{0\}$, there exists a unique
    $t_v > 0$ such that $t_v v \in \mathcal{N}_*$.  This $t_v$ is
    given by the explicit formula $t_v = \exp\bigl( - \int_\Omega v^2
    \ln\abs{v} \intd x / \abs{v}_2^2 \bigr)$.  Since $v \mapsto t_v$
    is continuous and $\mathcal{N}_*$ is the image of the unit sphere
    of $E_2$ under the map $v \mapsto t_v \, v$, $\mathcal{N}_*$ is
    compact.  Therefore, there exists a $v_* \in \mathcal{N}_*$ that
    achieves the minimum of $\E_*$ on $\mathcal{N}_*$.  Moreover, for
    all $v \in \mathcal{N}_*$ and all $t \ge 0$, $\E_*(tv) =
    \tfrac{1}{2} t^2 (1 - \ln t^2) \abs{v}_2^2 \le \E_*(v)$ so the
    reduced functional $\E_*$ possesses a Mountain-Pass structure.
  \end{enumerate}
\end{Rem}

\begin{proof}
  Propositions~\ref{weak} and~\ref{away} imply that
  $u_* \in E_2\setminus\{0\}$.
  Let $v \in E_2$.  We have
  \begin{align*}
    0 = \frac{1}{2 - p_n} \E'_{p_n}(u_{p_n})[v]
    &= \lambda_2 \int_\Omega
    \frac{u_{p_n} - \abs{u_{p_n}}^{p_n-2} u_{p_n}}{2 - p_n} \, v\\
    &= \lambda_2  \int_\Omega
    \frac{1}{2 - p_n}
    \int_2^{p_n} \ln\abs{u_{p_n}} \, \abs{u_{p_n}}^{q-2} u_{p_n} v
    \intd q \intd x
    \xrightarrow[n \to \infty]{}
    \lambda_2 \int_\Omega u_* \ln\abs{u_*} \, v.
  \end{align*}
  Thus $u_*$ satisfies~\eqref{eq:limit-equation}.
  Since $\E'_*(u)[v] = -2 \int_\Omega u \ln\abs{u}\, v$, $u_*$ is a critical
  point of $\E_*$ and, in particular, $u_* \in \mathcal{N}_*$.  It
  remains to show that $u_*$ achieves the minimal value of $\E_*$ on
  $\mathcal{N}_*$.

  Let $v \in \mathcal{N}_*$.
  Set $v_p := t^+_p v^+ + t^-_p v^-$ where $t^\pm > 0$ are the
  unique positive reals such that $v_p \in \mathcal{M}_p$ (they exist
  because $v$ changes sign, see page~\pageref{proj-nodal-Nehari}).
  Let $p_n \to 2$.
  Arguing as in the proof of Proposition~\ref{bound}, one can show
  that $(t_{p_n}^\pm)$ are bounded.  So, up to subsequences,
  $t_{p_n}^\pm \to t^\pm$ for some $t^\pm \in [0,\infty)$.
  Passing to the limit on equation~\eqref{eq:t+}
  and using~\eqref{eq:eigenfunctions+-}, one finds that
  $(t^- - t^+) \langle v^+, v^- \rangle = 0$ and so that
  $t^+ = t^-$.
  %
  In addition, as in the proof of Proposition~\ref{bound},
  we can also assume w.l.o.g.\ that
  \begin{equation}
    \label{eq:relation:t+t-}
    \forall n,\quad
    \delta_n :=
    t^+_{p_n} \bigl(
    \abs{v^+}_2^2 - (t_{p_n}^+)^{p_n-2} \abs{v^+}_{p_n}^{p_n} \bigr)
    = - t^-_{p_n} \bigl(
    \abs{v^-}_2^2 - (t_{p_n}^-)^{p_n-2} \abs{v^-}_{p_n}^{p_n} \bigr)
    \ge 0 .
  \end{equation}
  Using the fact that the bracket of the right expression is
  non-positive and passing to the limit (similarly to
  eq.~\eqref{eq:bound:t+}) yields
  \begin{equation*}
    t^+
    = t^- \ge \exp\biggl( -
    \frac{\int_\Omega \ln\abs{v^-} \, \abs{v^-}^2}{
      \abs{v^-}_2^2 } \biggr)
    > 0.
  \end{equation*}
  Thus $(t^\pm_{p_n})^{p_n - 2} \to 1$ and so $\delta_n \to 0$.
  Dividing~\eqref{eq:relation:t+t-} by $2 - p_n$ and passing to the
  limit gives
  \begin{equation}
    \label{eq:t+/t-}
    t^+ \Bigl( \ln t^+ \abs{v^+}_2^2
    + \int_\Omega \ln\abs{v^+} \, \abs{v^+}^2 \Bigl)
    = - t^- \Bigl( \ln t^- \abs{v^-}_2^2
    + \int_\Omega \ln\abs{v^-} \, \abs{v^-}^2 \Bigl)
  \end{equation}
  where we used the elementary identity $t^{p-2} \abs{v}_p^p - \abs{v}_2^2
  = \int_2^p t^{q-2} \ln t \, \abs{v}_q^q
  + t^{q-2} \int_\Omega \ln\abs{v}\, \abs{v}^q \intd x \intd q$,
  for all $v \in H$ and $t > 0$,
  to compute the limit.
  Since $t^+ = t^-$, \eqref{eq:t+/t-} can be rewritten
  \begin{equation*}
    \ln t^+ \, \abs{v}_2^2
    + \int_\Omega \ln\abs{v} \, \abs{v}^2  = 0.
  \end{equation*}
  Recalling that $v \in \mathcal{N}_*$ means $\int \ln\abs{v} \,
  \abs{v}^2 = 0$, one deduces that $t^+ = 1 = t^-$.
  Thus $v_p \to v$.

  Because $u_{p_n}$ has least energy on $\mathcal{M}_{p_n}$,
  $\tilde\E_{p_n}(u_{p_n}) \le \tilde\E_{p_n}(v_{p_n})$.
  Because $u_{p_n}$ and $v_{p_n}$ belong to $\mathcal{N}_{p_n}$, this
  is equivalent to $\abs{u_{p_n}}_{p_n}^{p_n}
  \le \abs{v_{p_n}}_{p_n}^{p_n}$.  Passing to the limit
  and using the fact that $u_*, v \in \mathcal{N}_*$ yield
  the desired inequality
  \begin{math}
    2 \E_*(u_*) = \abs{u_*}_2^2
    \le \abs{v}_2^2 = 2 \E_*(v)
  \end{math}.
\end{proof}

\section{Numerical examples}

In this section, we illustrate our results by numerical computations.
We consider
the particular case of  the fractional Laplacian
problem~\eqref{eq:pbmF} for some values of $s\in (0,1]$.
The functional and its derivatives are computed thanks to the Finite
Element Method.  Ground states (resp.\ least energy nodal solutions)
are approximated using the the Mountain-Pass Algorithm (resp.\ the
Modified Mountain-Pass Algorithm) \cite{MPA,neu2}.

Let us give some details on the computation of the various quantities.
Given a mesh of the domain $\Omega$, the integrals $\int_\Omega V
u^2$, $\int_\Omega \abs{u}^p$,... are approximated using standard
quadrature rules on each element of the mesh.  The hardest part for
evaluating the functional $\E_p$ and its derivatives is clearly the
computation of the stiffness matrix.  More precisely, if
$(\phi_i)_{i=1}^n$ denotes the usual FEM basis consisting of ``hat
functions'' for each interior node of the mesh, we need to compute
\begin{equation}
  \label{eq:stiffness}
  \langle \phi_i,\phi_j \rangle_H
  = \int_{\IR^N\times\IR^N} \bigl(\phi_i(x)-\phi_i(y)\bigr)
  \bigl(\phi_j(x)-\phi_j(y)\bigr) K(x-y)\intd (x, y).
\end{equation}
The two difficulties are that the kernel $K$ is singular and the
domain is unbounded.
The convergence of the finite element method for this type of
non-local operator was proved by Marta D'Elia and Max
Gunzburger~\cite{Elia-Gunzburger}.  In order to
compute~\eqref{eq:stiffness}, they restrict their attention to $N=1$,
use an ``interaction domain'' $\Omega_{\mathcal{I}} \subseteq \IR^N
\setminus \Omega$ and assume that $K$ vanishes outside a ball of
``large'' radius.  In this paper, we deal directly
with~\eqref{eq:stiffness} posed on $\IR^N\times\IR^N$ with $K(x) =
\tfrac{1}{2} c_{N,s} \abs{x}^{-N-2s}$.  The reason it is possible
results from a couple of remarks.  First notice that
\begin{equation*}
  \mathbf{S}_i := \supp\bigl((x,y) \mapsto \phi_i(y) - \phi_i(x)\bigr)
  = (\supp \phi_i \times \IR^N) \cup (\IR^N \times \supp\phi_i).
\end{equation*}
Thus the integral in~\eqref{eq:stiffness} has only to be considered on
$\mathbf{S}_i \cap \mathbf{S}_j$.  For brevity, let us write
$S_i := \supp\phi_i$.  Expanding $\mathbf{S}_i \cap \mathbf{S}_j$ and
remarking that the integral~\eqref{eq:stiffness} is unchanged if one
swaps $x$ and $y$, one deduces that
\begin{equation*}
  \label{eq:slipt-integral}
  \int_{\mathbf{S}_i \cap \mathbf{S}_j}
  = 2 \int_{(S_i \cap S_j) \times \complement(S_i \cup S_j)}
  + \int_{(S_i \cap S_j)^2}
  + 2 \int_{(S_i \setminus S_j) \times S_j}
  + 2 \int_{(S_i \cap S_j) \times (S_j \setminus S_i)}  .
\end{equation*}
Among these four sets, the sole unbounded one is $(S_i \cap S_j)
\times \complement(S_i \cup S_j)$.  If $\supp\phi_i \cap \supp\phi_j$
has zero Lebesgue measure, i.e., if $i$ and $j$ are not indices of
neighboring nodes, then the integral boils down to
\begin{equation*}
  \langle \phi_i,\phi_j \rangle_H
  = -2\int_{\supp\phi_i \times \supp\phi_j}
  \phi_i(x)\phi_j(y) K(x-y)\intd (x, y),
\end{equation*}
where $K(x-y)$ is non-singular except when $x = y \in \supp\phi_i \cap
\supp\phi_j$ (often empty) where both $\phi_i$ and $\phi_j$ vanish.
If $\supp\phi_i \cap \supp\phi_j$ has non-zero measure, then the
integral on the
unbounded set must be taken into account.  However, it simplifies to
\begin{multline}
  \label{eq:integ-exterior}
  \int_{(S_i \cap S_j) \times \complement(S_i \cup S_j)}
  \bigl(\phi_i(x)-\phi_i(y)\bigr)
  \bigl(\phi_j(x)-\phi_j(y)\bigr) K(x-y)\intd (x, y)\\
  = \int_{S_i \cap S_j}
  \phi_i(x) \phi_j(x)
  \biggl( \int_{\complement(S_i \cup S_j)} K(x-y)\intd y \biggr)\intd x,
\end{multline}
and therefore to estimate this integral it is enough to be able to
estimate the integral of $K$ in a neighborhood of infinity.

\begin{figure}[ht]
  \newcommand{\xmin}{-0.5}
  \begin{minipage}[b]{0.48\linewidth}
    \centering
    \begin{tikzpicture}[x=4ex, y=4ex]
      \newcommand{\xmax}{7}%
      \fill[color=red!20]
      (2,\xmin) -- (2,2) -- (\xmin, 2) -- (\xmin, 4.5) -- (2,4.5) -- (2,\xmax)
      -- (4.5,\xmax) -- (4.5,4.5) -- (\xmax,4.5) -- (\xmax, 2) -- (4.5,2)
      -- (4.5,\xmin) -- cycle;
      \fill[color=blue!30, semitransparent]
      (2,\xmin) -- (2,2) -- (\xmin, 2) -- (\xmin, 4.5) -- (2,4.5) -- (2,\xmax)
      -- (4.5,\xmax) -- (4.5,4.5) -- (\xmax,4.5) -- (\xmax, 2) -- (4.5,2)
      -- (4.5,\xmin) -- cycle;
      \foreach \t/\l in {2/i-1, 3/i, 4.5/i+1} {%
        \draw[color=gray!70] (\t,\xmin) -- (\t, \xmax);
        \draw[color=gray!70] (\xmin,\t) -- (\xmax, \t);
        \node[left] at (0,\t) {$x_{\l}$};
      }
      \node[below, xshift=0.5ex] at (2,0) {\rotatebox{-30}{$x_{i-1}$}};
      \node[below] at (3,0) {$x_{i}$};
      \node[below, xshift=1ex] at (4.5,0) {\rotatebox{-30}{$x_{i+1}$}};
      \node[below right] at (4.5,2) {$\mathbf{S}_i = \mathbf{S}_j$};

      \draw[->] (\xmin,0) -- (\xmax,0) node[below]{$x$};
      \draw[->] (0,\xmin) -- (0,\xmax) node[left]{$y$};
      \draw[dashed] (\xmin,\xmin) -- (\xmax, \xmax)
      node[pos=0.88, sloped, yshift=1ex]{$x=y$};
    \end{tikzpicture}
    \caption{Case $i = j$}
    \label{fig:supp i=j}
  \end{minipage}
  \hfill
  \begin{minipage}[b]{0.48\linewidth}
    \centering
    \begin{tikzpicture}[x=4ex, y=4ex]
      \newcommand{\xmax}{7}%
      \fill[color=red!20]
      (1.8,\xmin) -- (1.8,1.8) -- (\xmin, 1.8) -- (\xmin, 4.5)
      -- (1.8,4.5) -- (1.8,\xmax)
      -- (4.5,\xmax) -- (4.5,4.5) -- (\xmax,4.5) -- (\xmax, 1.8) -- (4.5,1.8)
      -- (4.5,\xmin) -- cycle;
      \fill[color=blue!30, semitransparent]
      (3,\xmin) -- (3,3) -- (\xmin, 3) -- (\xmin, 5.5) -- (3,5.5) -- (3,\xmax)
      -- (5.5,\xmax) -- (5.5,5.5) -- (\xmax,5.5) -- (\xmax, 3) -- (5.5,3)
      -- (5.5,\xmin) -- cycle;
      \foreach \t/\l in {1.8/i-1, 3/i, 4.5/i+1, 5.5/i+2} {%
        \draw[color=gray!70] (\t,\xmin) -- (\t, \xmax);
        \draw[color=gray!70] (\xmin,\t) -- (\xmax, \t);
        \node[left] at (0,\t) {$x_{\l}$};
      }
      \node[below, xshift=0.5ex] at (1.8,0) {\rotatebox{-30}{$x_{i-1}$}};
      \node[below] at (3,0) {$x_{i}$};
      \node[below, xshift=1ex] at (4.5,0) {\rotatebox{-30}{$x_{i+1}$}};
      \node[below, xshift=1ex] at (5.5,0) {\rotatebox{-30}{$x_{i+2}$}};
      \node[left] at (\xmax, 2.35) {$\mathbf{S}_i$};
      \node[left] at (\xmax, 4.95) {$\mathbf{S}_j$};

      \draw[->] (\xmin,0) -- (\xmax,0) node[below]{$x$};
      \draw[->] (0,\xmin) -- (0,\xmax) node[left]{$y$};
      \draw[dashed] (\xmin,\xmin) -- (\xmax, \xmax)
      node[pos=0.93, sloped, yshift=1ex]{$x=y$};
    \end{tikzpicture}
    \caption{Case $i + 1 = j$}
  \end{minipage}
\end{figure}

\begin{figure}[ht]
  \newcommand{\xmin}{-0.5}
  \begin{minipage}[b]{0.55\linewidth}
    \centering
    \begin{tikzpicture}[x=4ex, y=4ex]
      \newcommand{\xmax}{10}%
      \fill[color=red!20]
      (2,\xmin) -- (2,2) -- (\xmin, 2) -- (\xmin, 4.3) -- (2,4.3) -- (2,\xmax)
      -- (4.3,\xmax) -- (4.3,4.3) -- (\xmax,4.3) -- (\xmax, 2) -- (4.3,2)
      -- (4.3,\xmin) -- cycle;
      \fill[color=blue!30, semitransparent]
      (5.8,\xmin) -- (5.8,5.8) -- (\xmin, 5.8) -- (\xmin, 8)
      -- (5.8,8) -- (5.8,\xmax)
      -- (8,\xmax) -- (8,8) -- (\xmax,8) -- (\xmax, 5.8) -- (8,5.8)
      -- (8,\xmin) -- cycle;
      \foreach \t/\l in {2/i-1, 3/i ,4.3/i+1,  5.8/j-1, 7/j, 8/j+1} {%
        \draw[color=gray!70] (\t,\xmin) -- (\t, \xmax);
        \draw[color=gray!70] (\xmin,\t) -- (\xmax, \t);
        \node[left] at (0,\t) {$x_{\l}$};
      }
      \node[below, xshift=0.5ex] at (2,0) {\rotatebox{-30}{$x_{i-1}$}};
      \node[below] at (3,0) {$x_{i}$};
      \node[below, xshift=1ex] at (4.3,0) {\rotatebox{-30}{$x_{i+1}$}};
      \node[below, xshift=0.5ex] at (5.8,0) {\rotatebox{-30}{$x_{j-1}$}};
      \node[below] at (7,0) {$x_{j}$};
      \node[below, xshift=1ex] at (8,0) {\rotatebox{-30}{$x_{j+1}$}};
      \node[left] at (\xmax, 2.45) {$\mathbf{S}_i$};
      \node[left] at (\xmax, 6.3) {$\mathbf{S}_j$};

      \draw[->] (\xmin,0) -- (\xmax,0) node[below]{$x$};
      \draw[->] (0,\xmin) -- (0,\xmax) node[left]{$y$};
      \draw[dashed] (\xmin,\xmin) -- (\xmax, \xmax)
      node[pos=0.93, sloped, yshift=1ex]{$x=y$};
    \end{tikzpicture}
    \caption{Case $i + 1 < j$}
    \label{fig:supp i+1<j}
  \end{minipage}
\end{figure}

For the one-dimensional case ($N=1$) where $\Omega$ is an interval,
the mesh is simply given by points $x_1 < x_2 < \cdots < x_M$ such that
$\Omega = \mathopen]x_1, x_M\mathclose[$.
The various possibilities for the sets $\mathbf{S}_i \cap
\mathbf{S}_j$ are depicted in Fig.~\ref{fig:supp i=j}--\ref{fig:supp
  i+1<j}.  For $K(x) = \tfrac{1}{2} c_{1,s} \abs{x}^{-1-2s}$,
the integrals on the various rectangles or unbouded strips in
Fig.~\ref{fig:supp i=j}--\ref{fig:supp i+1<j},
amount to compute
\begin{equation}
  \label{eq:elem-integ}
  \int_a^b \int_c^d
  \frac{\sum_{i,j=0}^2 q_{ij} x^i y^j}{
    \abs{y-x}^p} \intd y \intd x
\end{equation}
where $-\infty \le a < b \le c < d \le +\infty$, and $p \in \IR$.
Note that, thanks to the symmetry w.r.t.\ the diagonal, one may 
only integrate on $\{ (x,y) \mid y \ge x\}$ and remove the absolute value.
It is tedious but elementary to explicitly compute
integrals of the type~\eqref{eq:elem-integ} and thus to have a precise
estimate of the stiffness matrix at a low cost.

As Marta D'Elia and Max Gunzburger~\cite{Elia-Gunzburger} did, one can
judge the convergence of the method by comparing the FEM solution to
the explicit solution to $(-\Delta)^s u = 1$ on $\Omega = B(0,R)$,
namely
\begin{equation}
  \label{eq:explicit-sol}
  u^*(x) = 2^{-2s} \frac{\Gamma(N/2)}{\Gamma(N/2+s) \Gamma(1+s)}
  \bigl( R^2 - \abs{x}^2 \bigr)^s,
  \qquad x \in B(0,R).
\end{equation}
For a given $s$, let us denote $u_M$ the FEM solution to $(-\Delta)^s
u = 1$ on a mesh with $M$ nodes.  Figure~\ref{fig:error} shows the
errors $\norm{u_M - u^*}_H$ and $\abs{u_M - u^*}_2$ as functions
of~$M$.  These graphs suggest that $\norm{u_M - u^*}_H = O(M^{-0.5})$
and $\abs{u_M - u^*}_2 = O(M^{-0.8})$.

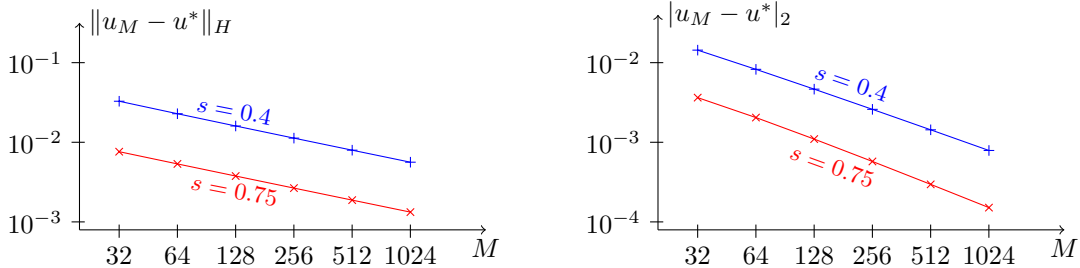
\begin{figure}[ht]
  \begin{minipage}[b]{0.48\linewidth}
  \centering
  \begin{tikzpicture}[x=17ex, y=7ex]
    \draw[->] (1.3,-3.1) -- +(2.1, 0) node[below]{$M$};
    \draw[->] (1.3,-3.1) -- +(0,2.6) node[right]{$\norm{u_M - u^*}_H$};
    \foreach \x/\l in {1.505/32, 1.806/64, 2.107/128, 2.408/256,
      2.709/512, 3.010/1024}{%
      \draw (\x, -3.1) +(0,+2pt) -- +(0,-3pt) node[below]{$\l$};
    }
    \foreach \y in {-1, -2, -3}{%
      \draw (1.3,\y) +(2pt,0) -- +(-3pt,0) node[left]{$10^{\y}$};
    }
    \draw[color=red] plot[mark=x] file {graphs/err_norm_0.75.dat};
    \node[color=red] at (2.1,-2.6) {\small \rotatebox{-11}{$s=0.75$}};
    \draw[color=blue] plot[mark=+] file {graphs/err_norm_0.4.dat};
    \node[color=blue] at (2.1,-1.6) {\small \rotatebox{-11}{$s=0.4$}};
  \end{tikzpicture}    
  \end{minipage}
  \hfill
  \begin{minipage}[b]{0.48\linewidth}
  \centering
  \begin{tikzpicture}[x=17ex, y=7ex]
    \draw[->] (1.3,-4.1) -- +(2.1, 0) node[below]{$M$};
    \draw[->] (1.3,-4.1) -- +(0,2.7) node[right]{$\abs{u_M - u^*}_2$};
    \foreach \x/\l in {1.505/32, 1.806/64, 2.107/128, 2.408/256,
      2.709/512, 3.010/1024}{%
      \draw (\x, -4.1) +(0,+2pt) -- +(0,-3pt) node[below]{$\l$};
    }
    \foreach \y in {-2, -3, -4}{%
      \draw (1.3,\y) +(2pt,0) -- +(-3pt,0) node[left]{$10^{\y}$};
    }
    \draw[color=red] plot[mark=x] file {graphs/err_norm_L2_0.75.dat};
    \node[color=red] at (2.2, -3.3) {\small \rotatebox{-20}{$s=0.75$}};
    \draw[color=blue] plot[mark=+] file {graphs/err_norm_L2_0.4.dat};
    \node[color=blue] at (2.3, -2.25) {\small \rotatebox{-19}{$s=0.4$}};
  \end{tikzpicture}    
  \end{minipage}

  \caption{Errors $\norm{u_M - u^*}_H$ and $\abs{u_M - u^*}_2$
    w.r.t.\ the number of nodes $M$.}
  \label{fig:error}
\end{figure}

Let us now turn to the non-linear problem~\eqref{eq:pbmF} with $V =
0$, $p = 4$ and $\Omega = \mathopen]-1, 1\mathclose[$.  The initial
function for the Mountain-Pass Algorithm (resp.\ the Modified
Mountain-Pass Algorithm) is $u_0(x) = \cos(\pi x / 2)$ (resp.\ $u_0(x)
= \sin(\pi x)$) and the algorithms stop when $\norm{\nabla \E_p}_H \le
10^{-2}$.  The ground state and l.e.n.s.\ are plotted in
Fig.~\ref{fig:sols1D} for several values of~$s$.  Some characteristics
of the solutions are given in Table~\ref{table:sols1D}.  Note that,
for $p$ fixed,
the smaller $s$ is, the more concentrated around~$0$ (resp.\ around
$\pm 1/2$) the ground state (resp.\ the l.e.n.s.)\ becomes.
This contrasts with the linear case $(-\Delta)^s u = 1$ where the
solution~\eqref{eq:explicit-sol} goes to~$1$ as $s \to 0$.

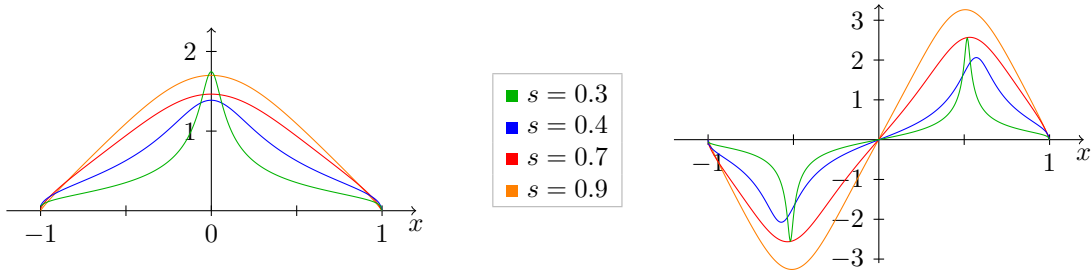
\begin{figure}[ht]
  \begin{minipage}{0.4\linewidth}
    \centering
    \begin{tikzpicture}[x=15ex, y=7ex]
      \draw[->] (-1.2,0) -- (1.2,0) node[below]{$x$};
      \draw[->] (0,0) -- (0, 2.3);
      \foreach \x in {-1, 0, 1}{%
        \draw (\x, 2pt) -- (\x,-2pt) node[below]{$\x$};
      }
      \draw (0.5, 2pt) -- (0.5, -2pt);
      \draw (-0.5, 2pt) -- (-0.5, -2pt);
      \foreach \y in {1, 2}{%
        \draw (2pt, \y) -- (-2pt, \y) node[left]{$\y$};
      }
      \draw[color=orange] plot file{graphs/u1D_0.9_4.dat};
      \draw[color=red] plot file{graphs/u1D_0.7_4.dat};
      \draw[color=blue] plot file{graphs/u1D_0.4_4.dat};
      \draw[color=darkgreen] plot file{graphs/u1D_0.3_4.dat};
    \end{tikzpicture}
  \end{minipage}
  \hfill
  \arrayrulecolor{gray!60}
  \begin{tabular}{|l|}
    \hline
    \textcolor{darkgreen}{\rule{1ex}{1ex}} $s = 0.3$\rule{0pt}{2.4ex}\\
    \textcolor{blue}{\rule{1ex}{1ex}} $s = 0.4$\\
    \textcolor{red}{\rule{1ex}{1ex}} $s = 0.7$\\
    \textcolor{orange}{\rule{1ex}{1ex}} $s = 0.9$\\
    \hline
  \end{tabular}
  \hfill
  \begin{minipage}{0.4\linewidth}
    \centering
    \begin{tikzpicture}[x=15ex, y=3.5ex]
      \draw[->] (-1.2,0) -- (1.2,0) node[below]{$x$};
      \draw[->] (0,-3.1) -- (0, 3.4);
      \foreach \x in {-1, 1}{%
        \draw (\x, 2pt) -- (\x,-2pt) node[below]{$\x$};
      }
      \draw (0.5, 2pt) -- (0.5, -2pt);
      \draw (-0.5, 2pt) -- (-0.5, -2pt);
      \foreach \y in {-3, -2, -1, 1, 2, 3}{%
        \draw (2pt, \y) -- (-2pt, \y) node[left]{$\y$};
      }
      \draw[color=orange] plot file{graphs/u1D_nodal_0.9_4.dat};
      \draw[color=red] plot file{graphs/u1D_nodal_0.7_4.dat};
      \draw[color=blue] plot file{graphs/u1D_nodal_0.4_4.dat};
      \draw[color=darkgreen] plot file{graphs/u1D_nodal_0.3_4.dat};
    \end{tikzpicture}
  \end{minipage}
  \caption{Ground state and l.e.n.s. for $s \in \{0.3, 0.4, 0.7, 0.9\}$.}
  \label{fig:sols1D}
\end{figure}

\begin{table}[!ht]
  \begin{tabular}{c cc ccc}
    $s$& $\E_4(u_1)$& $\max u_1$& $\E_4(u_2)$& $\max u_2$& $\min u_2$\\
    \hline
    $0.3$& 0.29& 1.7 &  \hphantom{1}0.74& $2.5$& $-2.5$\\
    $0.4$& 0.38& 1.4 &  \hphantom{1}1.41& $2.1$& $-2.1$\\
    $0.7$& 0.76& 1.5 &  \hphantom{1}6.45& $2.6$& $-2.6$\\
    $0.9$& 1.39& 1.7 &  18.30& $3.3$& $-3.3$\\
    \hline
    \\
  \end{tabular}
  \caption{Characteristics of the ground state $u_1$ and the
    l.e.n.s.\ $u_2$ for $\Omega = \mathopen]-1,1\mathclose[$.}
  \label{table:sols1D}
\end{table}

If one looks at the first and second eigenfunctions $\phi_1$ and
$\phi_2$ (see Fig.~\ref{fig:eigen1D}), the concentration
phenomena may be
surprising as one expects $u_1$ (resp.\ $u_2$) to resemble $\phi_1$
(resp.\ $\phi_2$).  However, the above results say that the latter is
true for $s$ fixed and $p \to 2$.  If one set $s$ to, say, $0.3$, and
let $p \to 2$, one clearly sees on Fig.~\ref{fig:p->2} that the ground
state goes to a multiple of $\phi_1$.

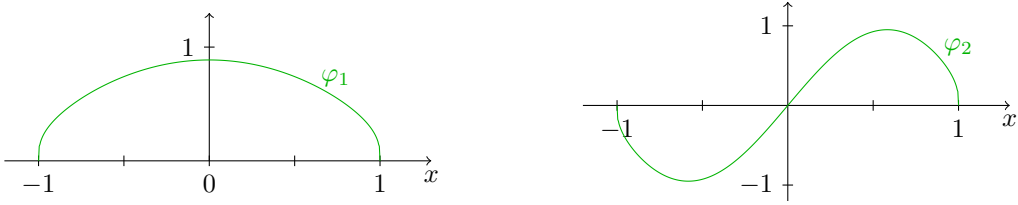
\begin{figure}[ht]
  \begin{minipage}[b]{0.48\linewidth}
    \centering
    \begin{tikzpicture}[x=15ex, y=10ex]
      \draw[->] (-1.2,0) -- (1.3,0) node[below]{$x$};
      \draw[->] (0,0) -- (0, 1.3);
      \foreach \x in {-1, 0, 1}{%
        \draw (\x, 2pt) -- (\x,-2pt) node[below]{$\x$};
      }
      \draw (0.5, 2pt) -- (0.5, -2pt);
      \draw (-0.5, 2pt) -- (-0.5, -2pt);
      \foreach \y in {1}{%
        \draw (2pt, \y) -- (-2pt, \y) node[left]{$\y$};
      }
      \draw[color=darkgreen] plot file{graphs/phi1_0.3.dat};
      \node[color=darkgreen] at (0.74, 0.73) {$\phi_1$};
    \end{tikzpicture}
  \end{minipage}
  \hfill
  \begin{minipage}[b]{0.48\linewidth}
    \centering
    \begin{tikzpicture}[x=15ex, y=7ex]
      \draw[->] (-1.2,0) -- (1.3,0) node[below]{$x$};
      \draw[->] (0,-1.2) -- (0, 1.3);
      \foreach \x in {-1, 1}{%
        \draw (\x, 2pt) -- (\x,-2pt) node[below]{$\x$};
      }
      \draw (0.5, 2pt) -- (0.5, -2pt);
      \draw (-0.5, 2pt) -- (-0.5, -2pt);
      \foreach \y in {-1, 1}{%
        \draw (2pt, \y) -- (-2pt, \y) node[left]{$\y$};
      }
      \draw[color=darkgreen] plot file{graphs/phi2_0.3.dat};
      \node[color=darkgreen] at (1, 0.73) {$\phi_2$};
    \end{tikzpicture}
  \end{minipage}
  \caption{First and second eigenfunctions for
    $s = 0.3$.
  }
  \label{fig:eigen1D}
\end{figure}

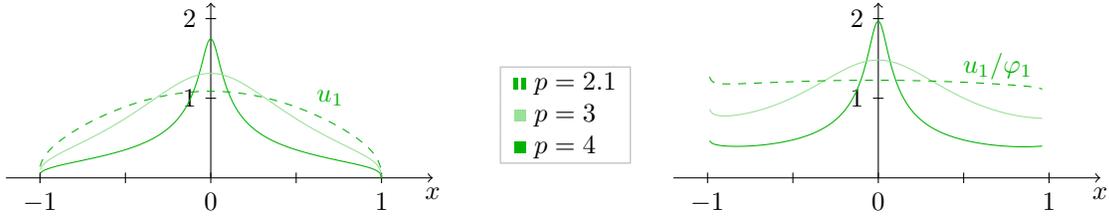
\begin{figure}[ht]
  \begin{minipage}[b]{0.4\linewidth}
    \centering
    \begin{tikzpicture}[x=15ex, y=7ex]
      \draw[->] (-1.2,0) -- (1.3,0) node[below]{$x$};
      \draw[->] (0,0) -- (0, 2.2);
      \foreach \x in {-1, 0, 1}{%
        \draw (\x, 2pt) -- (\x,-2pt) node[below]{$\x$};
      }
      \draw (0.5, 2pt) -- (0.5, -2pt);
      \draw (-0.5, 2pt) -- (-0.5, -2pt);
      \foreach \y in {1, 2}{%
        \draw (2pt, \y) -- (-2pt, \y) node[left]{$\y$};
      }
      \draw[color=darkgreen] plot file{graphs/u1D_0.3_4.dat};
      \draw[color=darkgreen!40] plot file{graphs/u1D_0.3_3.dat};
      \draw[color=darkgreen, dashed] plot file{graphs/u1D_0.3_2.1.dat};
      \node[color=darkgreen] at (0.7, 1) {$u_1$};
    \end{tikzpicture}
  \end{minipage}
  \hfill
  \arrayrulecolor{gray!60}
  \raisebox{5ex}{%
    \begin{tabular}[b]{|l|}
      \hline
      \textcolor{darkgreen}{\rule{0.35ex}{1ex}\hspace{0.3ex}\rule{0.35ex}{1ex}}
      $p = 2.1$\\
      \textcolor{darkgreen!40}{\rule{1ex}{1ex}} $p = 3$\\
      \textcolor{darkgreen}{\rule{1ex}{1ex}} $p = 4$\\
      \hline
    \end{tabular}}
  \hfill
  \begin{minipage}[b]{0.4\linewidth}
    \centering
    \begin{tikzpicture}[x=15ex, y=7ex]
      \draw[->] (-1.2,0) -- (1.3,0) node[below]{$x$};
      \draw[->] (0,0) -- (0, 2.2);
      \foreach \x in {-1, 0, 1}{%
        \draw (\x, 2pt) -- (\x,-2pt) node[below]{$\x$};
      }
      \draw (0.5, 2pt) -- (0.5, -2pt);
      \draw (-0.5, 2pt) -- (-0.5, -2pt);
      \foreach \y in {1, 2}{%
        \draw (2pt, \y) -- (-2pt, \y) node[left]{$\y$};
      }
      \draw[color=darkgreen] plot file{graphs/u1D_phi1_0.3_4.dat};
      \draw[color=darkgreen!40] plot file{graphs/u1D_phi1_0.3_3.dat};
      \draw[color=darkgreen, dashed] plot file{graphs/u1D_phi1_0.3_2.1.dat};
      \node[color=darkgreen] at (0.7, 1.4) {$u_1/\phi_1$};
    \end{tikzpicture}
  \end{minipage}
  \caption{Comparison of the ground state $u_1$ with $\phi_1$
    for $s = 0.3$ and $p \in \{2.1, 3, 4\}$.}
  \label{fig:p->2}
\end{figure}

\medskip

For the two dimensional case, the computation of the stiffness
matrix~\eqref{eq:stiffness}
is more challenging~\cite[p.~1259]{Elia-Gunzburger}.  The reason is
that there are no longer explicit formulas for the integrals and
$\complement(S_i \cup S_j)$ is not a simple shape.
Let us give some information on how we estimate the 
stiffness matrix~\eqref{eq:stiffness}.
The functions of the space~$H$ are approximated by $P^1$-finite
elements on a triangular mesh~$\T$ of~$\Omega$ (i.e., continuous
functions that are affine on each triangle of the mesh~$\T$).
We require that these functions vanish on
(the piecewise affine approximation of)~$\partial\Omega$.

To deal with the singular kernel, we use a generalized Duffy
transformation.  Let us explain how it works to compute
$\int_{T}\int_T \bigl(\phi_i(x)-\phi_i(y)\bigr)
\bigl(\phi_j(x)-\phi_j(y)\bigr) K(x-y)\intd x \intd y$ where $T$ is a
triangle of the mesh of~$\T$.  For the outer integral, we use a
standard second order integration scheme which evaluates the function
at the middle of the edges of~$T$.  For the inner one, we first make
use of the fact that $\phi_i$ (as well as $\phi_j$) is affine on~$T$
so that $\phi_i(x)-\phi_i(y) = \nabla\phi_i \cdot (x-y)$ where
$\nabla\phi_i$ is constant on~$T$, so the integral boils down to
\begin{equation}
  \label{eq:singular-integ}
  \tfrac{1}{2} c_{N,s}
  \int_T  \nabla\phi_i \cdot e_x \, \nabla\phi_j \cdot e_x
  \, \frac{1}{\abs{x-y}^{2s}} \intd x
  \qquad
  \text{where } e_x := \frac{x-y}{\abs{x-y}}.
\end{equation}
For each $y \in \partial T$ considered for the outer integral
approximation, one can project orthogonally $y$ on the two opposite
sides of~$T$ and compute the integral on $T$ as a sum or difference
(depending on whether the projection falls or not inside~$T$) of
integrals on right triangles $y q_3 p_2$, $y q_3 p_1$,
$y q_2 p_1$ and $y q_2 p_3$ (see Fig.~\ref{fig:splitting-T}).
It thus reamains to compute~\eqref{eq:singular-integ} on a right
triangle to which $y$ is a non-right corner.  So let $T$ be the
triangle $yqp$ with a right angle at $q$.  We perform the
following change of variable, dubbed generalized Duffy transformation,
\begin{equation*}
  x = y + u^\beta (q - y) + u^\beta v (p - q),
  \qquad
  (u,v) \in (0,1)^2,
\end{equation*}
so that \eqref{eq:singular-integ} becomes
\begin{equation*}
  c_{N,s} \abs{T}
  \int_0^1 \!\!\int_0^1
  \nabla\phi_i \cdot e_v \, \nabla\phi_j \cdot e_v
  \frac{\beta u^{2\beta(1-s) -1}}{(\abs{q-y}^2 + v^2 \abs{p-q}^2)^s}
  \intd u \intd v
  \quad
  \text{where } e_v :=
  \frac{q-y + v (p-q)}{\sqrt{\abs{q-y}^2 + v^2 \abs{p-q}^2}}
\end{equation*}
and $\abs{T}$ denotes the area of~$T$.  Taking $\beta := 1/(2(1-s))$,
so that $u^{2\beta(1-s) -1} \equiv 1$, gives a smooth integrand so
that the intégral can be estimated by standard means.

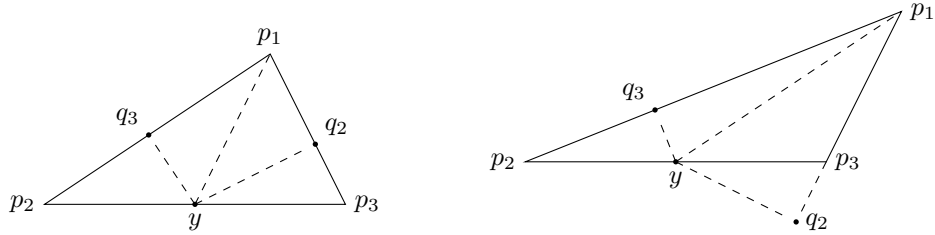
\begin{figure}[ht]
  \centering
  \begin{tikzpicture}
    \draw (0,0) node[left]{$p_2$} -- (4,0) node[right]{$p_3$}
    -- (3,2) node[above]{$p_1$} -- cycle;
    \fill (2,0) circle(1pt) node[below]{$y$}
    (1.384615384, 0.9230769) node(q3){} circle(1pt) node[above left]{$q_3$}
    (3.6, 0.8) node(q2){} circle(1pt) node[above right]{$q_2$};
    \draw[dashed] (2,0) -- (3,2)  (2,0) -- (q3)  (2,0) -- (q2);
  \end{tikzpicture}
  \hfil
  \begin{tikzpicture}
      \draw (0,0) node[left]{$p_2$} -- (4,0) node[right]{$p_3$}
      -- (5,2) node[right]{$p_1$} -- cycle;
    \fill (2,0) circle(1pt) node[below]{$y$}
    (1.7241379, 0.689655) node(q3){} circle(1pt) node[above left]{$q_3$}
    (3.6, -0.8) node(q2){} circle(1pt) node[right]{$q_2$};
    \draw[dashed] (2,0) -- (5,2)  (2,0) -- (q3)  (2,0) -- (q2);    
    \draw[dashed] (4,0) -- (q2);
  \end{tikzpicture}
  \caption{Splitting the integral on right triangles}
  \label{fig:splitting-T}
\end{figure}

To deal with the unboundedness of $\complement(S_i \cup S_j)$ in
\eqref{eq:integ-exterior}, we first integrate on $\Omega \setminus
(S_i \cup S_j)$.  Then, for each $x$ used to compute the outer
integral of~\eqref{eq:integ-exterior}, we mesh $B(x, R) \setminus
\Omega$, where $R$ is large enough so that $B(x, R) \supset \Omega$.
The integral on $B(x, R) \setminus \Omega$ is computed using that
mesh.  For the remaining set, $\complement B(x,R)$, the integral is
computed explicitly:
\begin{equation*}
  \int_{\complement B(x,R)} K(x-y) \intd y
  = \tfrac{1}{2} c_{N,s}
  \int_{\IS^{N-1}} \intd \theta \int_R^\infty
  r^{-N-2s} r^{N-1} \intd r
  = c_{N,s} \,\abs{\IS^{N-1}}\, \frac{1}{4 s R^{2s}}.
\end{equation*}
Note that this approach could be extended to kernels that are well
approximated by functions ``of separated variables'' in a neighborhood
of infinity:
$K(x) \approx \Theta(\theta) / r^{N+2s}$
when $\abs{x} = r \to +\infty$.

In Figures~\ref{fig:2D:s1}--\ref{fig:2D:s2}, you can see the computed
ground state and least energy nodal solutions for $s \in \{0.6, 0.9\}$
on the unit ball $B(0,1)$ for $p=4$.  The behavior is similar to the
one-dimensional case, namely the gound state is rotationally invariant
and the least energy nodal solution looks Schwarz foliated symmetric.
Moreover, both solutions concentrate as $s$ becomes smaller.

\begin{figure}[ht]
  \centering
  \begin{minipage}[b]{0.45\linewidth}
    \includegraphics[width=\linewidth, bb=70 215 552 580]{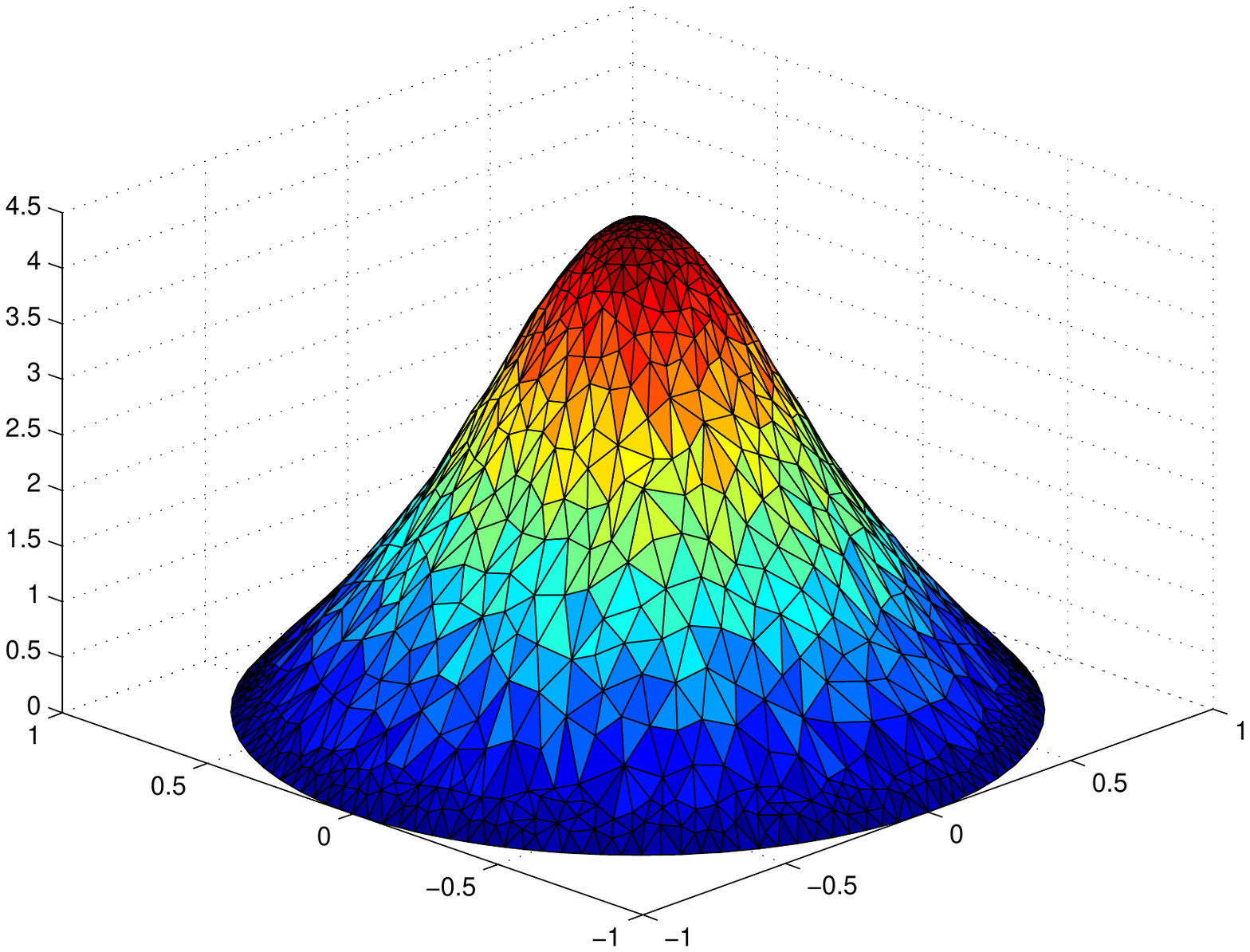}
  \end{minipage}
  \hfil
  \begin{minipage}[b]{0.45\linewidth}
    \includegraphics[width=\linewidth, bb=70 215 552 580]{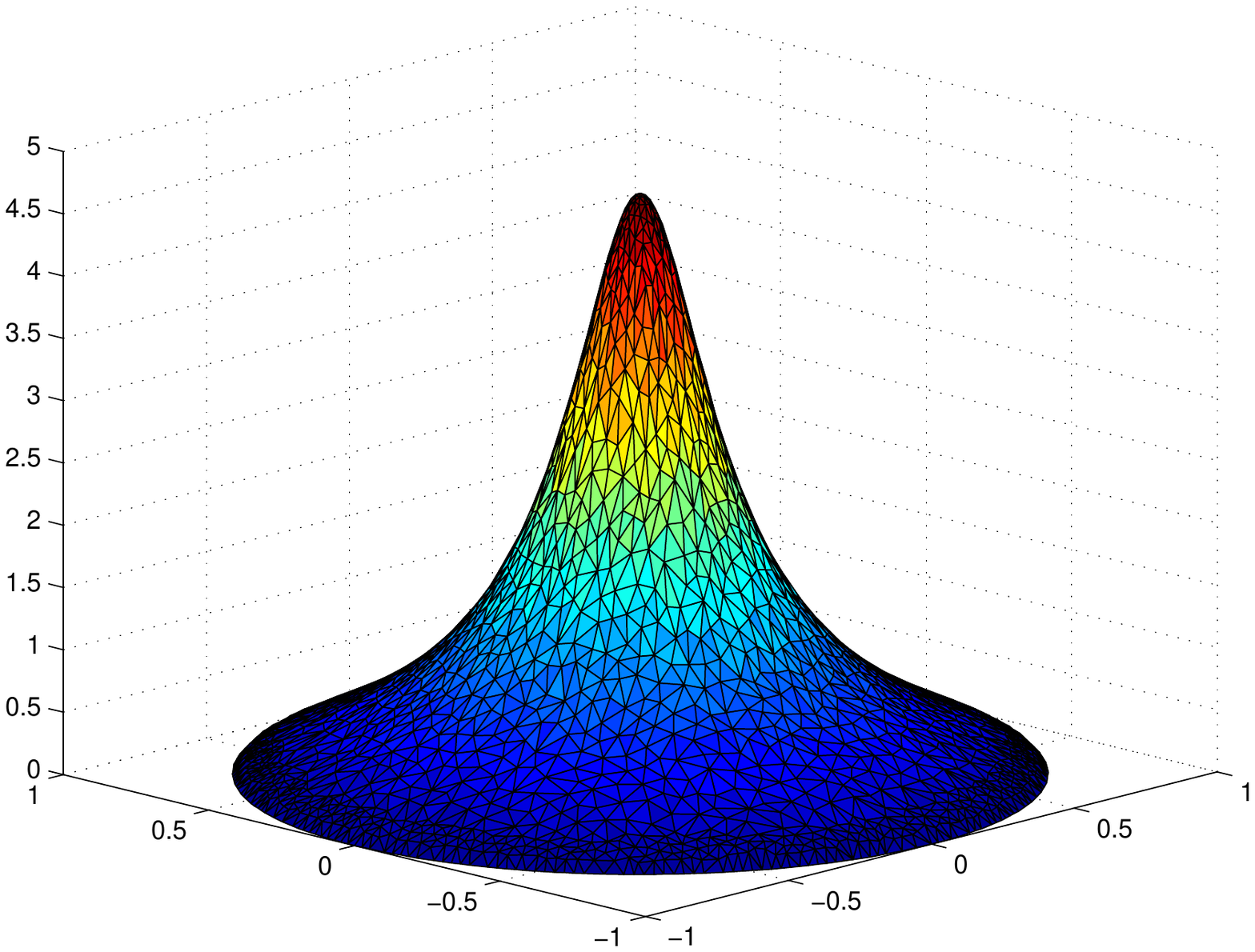}
  \end{minipage}
  \caption{Ground state solution for $s=0.9$ (left) and $s=0.6$ (right).}
  \label{fig:2D:s1}
\end{figure}

\begin{figure}[ht]
  \centering
  \begin{minipage}[b]{0.45\linewidth}
    \includegraphics[width=\linewidth, bb=72 214 553 580]{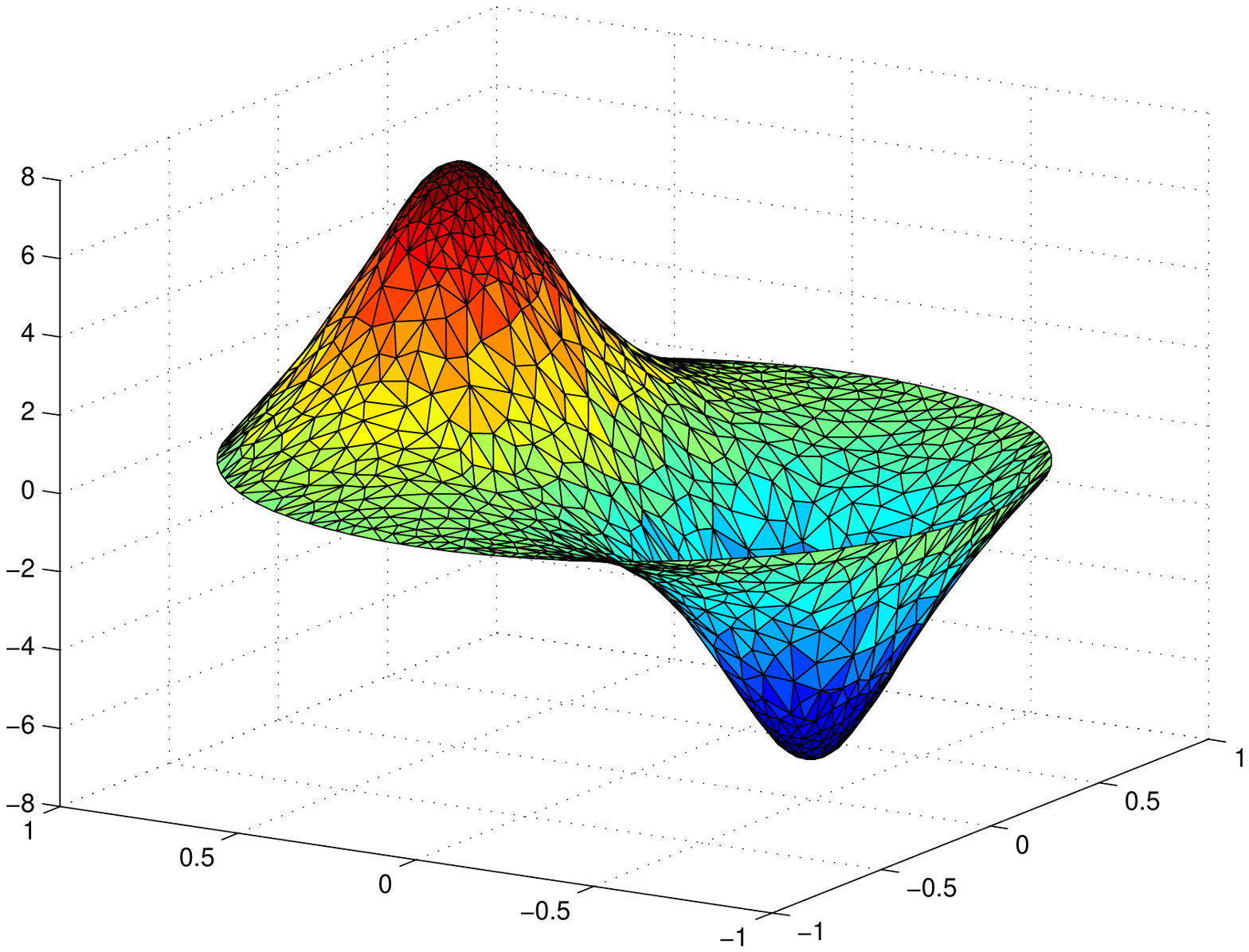}
  \end{minipage}
  \hfil
  \begin{minipage}[b]{0.45\linewidth}
    \includegraphics[width=\linewidth, bb=71 213 559 580]{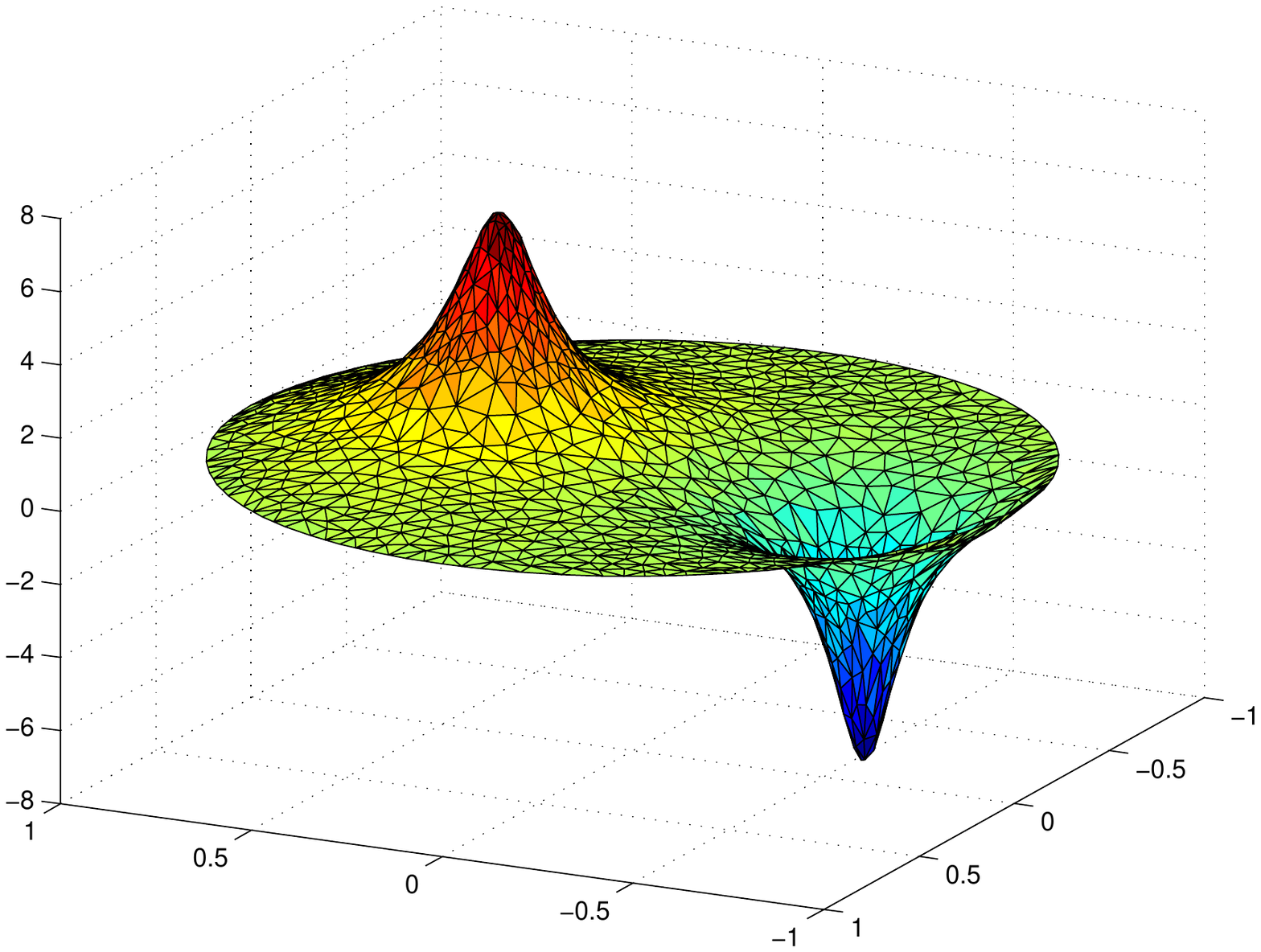}
  \end{minipage}
  \caption{Least energy nodal solution for $s=0.9$ (left)
    and $s = 0.6$ (right).}
  \label{fig:2D:s2}
\end{figure}

\bigskip
\bibliographystyle{amsplain}
\bibliography{fractionallaplacian}

\end{document}